\documentclass[3p,11pt]{elsarticle}
\usepackage{graphicx,subfig}
\usepackage{stmaryrd}
\usepackage{commath}
\usepackage{url}
\usepackage{amsmath}
\usepackage{amssymb}
\usepackage{amsthm}
\usepackage{mathrsfs}
\usepackage[numbers]{natbib}
\usepackage[hidelinks]{hyperref}
\hypersetup{
	colorlinks   = true, 
	urlcolor     = red, 
	linkcolor    = red, 
	citecolor   = red 
}
\usepackage{cleveref}

\DeclareFontFamily{U}{matha}{\hyphenchar\font45}
\DeclareFontShape{U}{matha}{m}{n}{
	<5> <6> <7> <8> <9> <10> gen * matha
	<10.95> matha10 <12> <14.4> <17.28> <20.74> <24.88> matha12
}{}
\DeclareSymbolFont{matha}{U}{matha}{m}{n}
\DeclareFontSubstitution{U}{matha}{m}{n}

\DeclareFontFamily{U}{mathx}{\hyphenchar\font45}
\DeclareFontShape{U}{mathx}{m}{n}{
	<5> <6> <7> <8> <9> <10>
	<10.95> <12> <14.4> <17.28> <20.74> <24.88>
	mathx10
}{}
\DeclareSymbolFont{mathx}{U}{mathx}{m}{n}
\DeclareFontSubstitution{U}{mathx}{m}{n}

\DeclareMathDelimiter{\vvvert}{0}{matha}{"7E}{mathx}{"17}

\makeatletter
\def\ps@pprintTitle{%
	\let\@oddhead\@empty
	\let\@evenhead\@empty
	\def\@oddfoot{}%
	\let\@evenfoot\@oddfoot}

\newcommand{\tnorm}[1]{{\left\vvvert #1 \right\vvvert}}
\newcommand{\jump}[1]{\llbracket #1 \rrbracket}
\newcommand{\av}[1]{\{\!\!\{#1\}\!\!\}}

\usepackage{xspace,color}

\newtheorem{thm}{Theorem}[section]
\newtheorem{lem}{Lemma}[section]

\newtheorem{cor}{Corollary}[section]

\theoremstyle{definition}
\newtheorem{defn}{Definition}[section]

\theoremstyle{remark}
\newtheorem{rem}{Remark}[section]

\begin{document}
	
\begin{frontmatter}	
	\title{Analysis of an exactly mass conserving space-time hybridized
		discontinuous Galerkin method for the time-dependent Navier--Stokes
		equations}
	
	\author[KLAK]{Keegan L. A. Kirk\fnref{label1}}
	\ead{k4kirk@uwaterloo.ca}
	\fntext[label1]{\url{https://orcid.org/0000-0003-1190-6708}}
	\address[KLAK]{Department of Applied Mathematics, University of
		Waterloo, Canada} 
	
	\author[TLH]{Tam\'as L. Horv\'ath\fnref{label2}}
	\ead{thorvath@oakland.edu}
	\fntext[label2]{\url{https://orcid.org/0000-0001-5294-5362}}
	\address[TLH]{Department of Mathematics and Statistics, Oakland
		University, U.S.A.}
	
	\author[KLAK]{Sander Rhebergen\fnref{label3}}
	\ead{srheberg@uwaterloo.ca}
	\fntext[label3]{\url{https://orcid.org/0000-0001-6036-0356}}  
	
	\begin{abstract}
		We introduce and analyze a space-time hybridized discontinuous Galerkin 
method for the evolutionary Navier--Stokes equations. Key features of the 
numerical scheme include point-wise mass conservation, energy stability, 
and pressure robustness. We prove that there exists a solution to the 
resulting nonlinear algebraic system in two and three spatial dimensions, 
and that this solution is unique in two spatial dimensions under a small 
data assumption. A priori error estimates are derived for the 
velocity in a mesh-dependent energy norm.
	\end{abstract}

	
\end{frontmatter}

	\section{Introduction}
	\label{sec:introduction}
	
	In this article, we are concerned with the numerical solution of the transient 
	Navier--Stokes 
	system posed on a convex polygonal ($d=2$) or polyhedral ($d=2$) domain 
	$\Omega \subset \mathbb{R}^d$, $d \in \cbr{2,3}$: given a suitably chosen body 
	force $f$, kinematic viscosity $0 <\nu \le 1$, and 
	initial data $u_0$, find $(u,p)$ such that
	\begin{subequations} \label{eq:ns_equations}
		\begin{align}
			\label{eq:ns_equations_a} \partial_t u - \nu \Delta u + \nabla \cdot( u 
			\otimes u) + \nabla p &= f, && 
			\hspace{-10mm}
			\text{in } \Omega\times (0,T], \\
			\nabla \cdot u &= 0, && \hspace{-10mm}\text{in } \Omega\times (0,T], \\
			u &=0, && \hspace{-10mm}\text{on }  \partial \Omega\times (0,T], \\
			u(x,0) &= u_0(x), && \hspace{-10mm}\text{in } \Omega.
		\end{align}
	\end{subequations}
	
	We analyze a space-time hybridizable 
	discontinuous Galerkin 
	(HDG) scheme for the evolutionary Navier--Stokes system \cref{eq:ns_equations} 
	based on the exactly 
	mass conserving, pointwise solenoidal discretization
	of Rhebergen and Wells \cite{Rhebergen:2018a}. Our analysis shows that the 
	resulting nonlinear algebraic system of equations is solvable for $d \in \cbr{2,3}$, uniquely solvable if $d=2$ under mild 
	conditions on the problem data, and that the numerical method converges 
	optimally in a mesh dependent norm. Moreover, the velocity error estimates are 
	independent of the pressure and inverse powers of $\nu$.

	The article is organized as follows: in the remainder of 
	\Cref{sec:introduction}, we give an overview of relevant literature, set 
	notation, introduce the space-time HDG 
	discretization, and introduce our main results. In \Cref{sec:preliminaries}, we 
	study the conservation properties of the numerical scheme and introduce 
	analysis tools that we will require. We consider the well-posedness of the  
	nonlinear algebraic system arising from the
	numerical scheme in \Cref{sec:well_posedness}. \Cref{sec:vel_err} is dedicated 
	to the error analysis for the velocity. We present a numerical test case with a manufactured solution in 
	\Cref{sec:numerical_results} to verify the theory, and finally we draw 
	conclusions in \Cref{sec:conclusion}. 
	
	\subsection{Related results}
	
	Pressure-robust discretizations for the solution of the Stokes and 
	Navier--Stokes equations have garnered much recent interest. These methods
	mimic at the discrete level the fundamental invariance property of
	incompressible flows that perturbing the external forcing term by a 
	gradient
	field affects only the pressure, and not the
	velocity~\cite{John:2017}. A consequence of this invariance property is that 
	the velocity error estimates are independent of
	the best approximation error of the pressure scaled by the inverse of
	the viscosity, which is in contrast to non-pressure-robust finite element 
	methods for
	incompressible flows. Classic examples of non-pressure-robust finite element 
	methods include the Taylor--Hood finite
	element~\cite{Hood:1974}, Crouzeix--Raviart~\cite{Crouzeix:1973}, MINI
	elements~\cite{Arnold:1984}, and discontinuous Galerkin
	methods~\cite{Cockburn:2002, Cockburn:2003, Pietro:book}.
	
	
	The use of the discontinuous Galerkin method as a higher-order time stepping 
	scheme has had much success in application to both parabolic and hyperbolic 
	equations 
	\cite{Chrysafinos:2006,Chrysafinos:2008,Chrysafinos:2010,Feistauer:2007, 
		Feistauer:2011, Dolejsi:book, Rhebergen:2013b, Sosa:2020,
		Thomee:book,
		Vegt:2002}. We 
	refer to the book by Thom\'{e}e 
	\cite{Thomee:book} for an introduction to the subject. Moreover, when combined with the 
	discontinuous Galerkin method in space, the resulting 
	space-time discretization is fully conservative, allows for hp-additivity in 
	space-time, and can achieve higher-order accuracy in both space and time 
	\cite{Feistauer:2007,Feistauer:2011,Vegt:2002}. 
	Furthermore, the unified treatment of the spatial and temporal discretizations 
	makes it an excellent candidate for the solution of partial differential 
	equations on time-dependent domains; see e.g. 
	\cite{Horvath:2019,Horvath:2020, Kirk:2019, 
		Rhebergen:2012,Rhebergen:2013b,Sudirham:2008} 
	for further developments 
	along this direction.

	Despite its advantages, use of the discontinuous Galerkin method in both 
	the spatial and temporal discretizations introduces a significant 
	computational burden compared to traditional time stepping methods. In 
	particular, the number of globally coupled degrees of freedom is 
	$\mathcal{O}(p^{d+1})$, where $p$
	denotes the polynomial order and $d$ is the spatial dimension of the
	problem under consideration. However, 
	the number of globally coupled degrees of freedom may be reduced through hybridization \cite{Cockburn:2009}. Approximate traces of the
	solution on element facets are introduced as new unknowns in the problem. The resulting linear
	system may then be reduced through static condensation to a global
	system of size $\mathcal{O}(p^{d})$ for only these approximate traces. This 
	strategy was employed by Rhebergen and
	Cockburn~\cite{Rhebergen:2012,Rhebergen:2013} who introduced the
	space-time hybridizable DG (HDG) method to alleviate the computational burden
	of space-time DG.
	
	For DG time stepping in the context of incompressible flow problems, we refer 
	to 
	\cite{Ahmed:2020,Chrysafinos:2010,Sudirham:2008,Rhebergen:2012,Rhebergen:2013b,Horvath:2019,Horvath:2020}.
	The references \cite{Ahmed:2020,Chrysafinos:2010} introduce DG time stepping 
	schemes combined with inf-sup stable conforming finite element methods for 
	incompressible flow problems. The references \cite{Sudirham:2008,Rhebergen:2013b} introduce space-time DG schemes for incompressible flows on time-dependent domains, while 
	\cite{Rhebergen:2012,Horvath:2019,Horvath:2020} introduce space-time HDG 
	schemes for 
	incompressible flows on time-dependent domains. Of 
	particular importance is the work of Chrysafinos and Walkington 
	\cite{Chrysafinos:2010}, wherein fine 
	properties of polynomials are used
	to develop tools essential to our analysis.

	\subsection{Notation}
	
	We use standard notation for Lebesgue and Sobolev spaces: given a
	bounded measurable set $D$, we denote by $L^p(D)$ the space of
	$p$-integrable functions. When $p=2$, we denote the $L^2(D)$ inner
	product by $(\cdot,\cdot)_D$. We denote by $W^{k,p}(D)$ the Sobolev
	space of $p$-integrable functions whose distributional derivatives up to order $k$ are
	$p$-integrable.  When $p=2$, we write $W^{k,p}(D) = H^k(D)$. We denote
	by $\gamma:H^{s+1/2}(D) \to H^s(\partial D)$ is the trace operator. We
	define $H_0^1(D)$ to be the subspace of $H^1(D)$ of functions with
	vanishing trace on the boundary of $D$.  We denote the space of
	polynomials of degree $k\ge0$ on $D$ by $P_k(D)$. We use 
	standard notation for spaces of vector valued functions with $d$ components, e.g. $L^2(D)^d$, $H^k(D)^d$, $P^k(D)^d$, etc.
	At times we drop the superscript for convenience, e.g. we denote by $\norm{\cdot}_{L^2(\Omega)}$  
	the norm on both $L^2(\Omega)$ and $L^2(\Omega)^d$.
	
	Next, let $U$ be a Banach space, $I = [a,b]$ an interval in $\mathbb{R}$, and $1 \le p < \infty$.
	We denote by
	$L^p(I;U)$ the Bochner space of $p$-integrable functions defined on 
	$I$ taking values in $U$. 
	When $p = \infty$, we denote by $L^{\infty}(I;U)$ the Bochner space of 
	essentially bounded functions taking values in $U$ and by $C(I;U)$ the space
	of (time) continuous functions taking values in $U$. By $H^k(I;U)$, we 
	denote the Bochner-Sobolev space for $k \ge 1$:
	\begin{equation*}
		H^k(I;U) := \cbr{u \in L^2(I;U) \; | \; \od[j]{u}{t} \in L^2(I;U), \; j 
			= 1, \dots , k},
	\end{equation*}
	equipped with its usual norm.
	Finally, we let
	$P_k(I;U)$ denote the space of 
	polynomials of degree $k \ge 0$ in time taking values in $U$.
	
	\subsection{The continuous problem}
	
	We will begin our discussion of the Navier--Stokes system with the theory of 
	weak solutions \cite{Temam:book}. 
	We define two function spaces, $H$ and $V$, as follows:
	\begin{subequations}
		\begin{align}
			\label{eq:defH}
			H &= \cbr[1]{u \in L^2(\Omega)^d \; | \; \nabla \cdot u = 0 \text{ and } u 
				\cdot n|_{\partial \Omega} = 0}, \\
			V &= \cbr[1]{u \in H_0^1(\Omega)^d \; | \; \nabla \cdot u = 0}.    
		\end{align}
	\end{subequations}
	Note that
	$H \subset H(\text{div};\Omega):=\cbr{v \in L^2(\Omega)^d\,|\, \nabla \cdot
		u \in L^2(\Omega)}$. We equip $H$ and $V$ with the standard norms on $L^2(\Omega)^d$ and $H_0^1(\Omega)^d$, respectively.
	By testing \cref{eq:ns_equations_a} with $v\in V$ and integrating by parts, we arrive at the following abstract ODE: find $u \in L^2(0,T;V) \cap 
	L^{\infty}(0,T;H)$ such that for a.e. $t \in (0,T]$,
	\begin{subequations} \label{eq:weak_sol}
		\begin{align}
			\langle \od{u}{t}, v \rangle_{V'\times V} + \nu(\nabla u, \nabla v) \; + ((u 
			\cdot 
			\nabla)u, v)  &= 
			\langle f, v \rangle_{V' \times V}, \quad \forall v \in V, \\
			u(0)&= u_0.
		\end{align}
	\end{subequations}
	\begin{rem}[On the regularity of weak solutions and consistency]
		We require at least $(u,p) \in 
		H^1(0,T;L^2(\Omega)^d) 
		\cap 
		L^2(0,T; H^{\frac{3}{2} + \epsilon}(\Omega)^d\cap V) \times L^2(0,T;H^1(\Omega)) \cap 
		L^2(0,T;L^2_0(\Omega)) $ for the consistency of the space-time HDG method. Thus, we restrict our attention to \emph{strong solutions}. We assume that (at least) $f \in 
		L^2(0,T;H)$ 
		and $u_0 \in V$. Provided
		the problem data is sufficiently small, the following result on strong solutions taken from \cite[Theorem 
		5.4]{Chrysafinos:2010} holds:
	\end{rem}
	\begin{thm} \label{thm:strong_solution}
		Let $\Omega \subset \mathbb{R}^3$ be a convex polyhedral domain. There exists a $C>0$, dependent on the final time $T$, such that for $u_0 \in V$ and $f \in L^2(0,T;H)$ satisfying
		\begin{equation}\label{eq:small_data_1}
			\norm{u_0}_{V}^2 + \frac{1}{\nu} \norm{f}_{L^2(0,T;L^2(\Omega))}^2 \le C\nu^2,
		\end{equation}
		there exists a unique \emph{strong solution} of \cref{eq:weak_sol} with $(u,p) \in L^{\infty}(0,T;V) \cap 
		L^2(0,T;H^2(\Omega)^d\cap V) \times L^2(0,T;H^1(\Omega)) \cap 
		L^2(0,T;L_0^2(\Omega))$ and $\partial_t u \in L^2(0,T;H)$ such that
		\begin{equation} \label{eq:energ_est_strong}
			\norm{u}_{L^{\infty}(0,T;V)}^2 + \nu 
			\norm{u}_{L^2(0,T;H^2(\Omega))}^2
			\lesssim \nu^2, \quad 	\norm{\partial_t u}^2_{L^2(0,T;L^2(\Omega))} \lesssim \nu^3.
		\end{equation}
	\end{thm}
	The assumption on the problem data \cref{eq:small_data_1} can be interpreted as 
	small initial data and body force, or large viscosity and arbitrary data. In light of \Cref{thm:strong_solution}, we will make the following assumption 
	on the problem data:

	\noindent
	\emph{Assumption 1.} 
	We assume that \cref{eq:small_data_1} holds. Note that, since $u_0 \in V 
	\subset H_0^1(\Omega)^d$, \cref{eq:small_data_1} implies the existence of a 
	constant $C>0$ such that
	\begin{equation}
		\label{eq:small_data_2}
		\norm{u_0}_{L^2(\Omega)}^2 + \frac{1}{\nu} 
		\norm{f}_{L^2(0,T;L^2(\Omega))}^2 \le C\nu^2.
	\end{equation}

	\begin{rem}
		If $\Omega \subset \mathbb{R}^2$ is a convex polygon, the existence of a global unique
		strong solution $(u,p)$ can be shown without any restriction on the problem data (see e.g. \cite{Temam:book}).
		However, we will later require a similar restriction on the data 
		to prove the uniqueness of the discrete solution in two dimensions. 
		We therefore assume \cref{eq:small_data_1} \emph{even in the two dimensional case}.
	\end{rem}
	Therefore, given $f \in L^2(0,T;H)$ and $u_0 \in V$ satisfying the small data 
	assumption \cref{eq:small_data_2}, we consider the following space-time 
	formulation for the strong
	solution to the Navier--Stokes system: for all $(v,q) \in L^2(0,T;H_0^1(\Omega)^d) \cap H^1(0,T;L^2(\Omega)^d) \times L^2(0,T;L_0^2(\Omega) \cap H^1(\Omega))$,
	find $(u,p) 
	\in L^{\infty}(0,T;V) \cap 
	L^2(0,T;H^2(\Omega)^d\cap V) \cap H^1(0,T;H) \times L^2(0,T;L_0^2(\Omega) \cap 
	H^1(\Omega))$ satisfying
	\begin{equation} \label{eq:continuous_st_problem}
		\begin{split}
				-&\int_0^T ( u, \partial_t v) \dif t + 
			\int_0^T((u \cdot \nabla)u , v) \dif t + \nu \int_0^T (\nabla u, 
			\nabla 
			v) \dif t + \int_0^T (\nabla p , v)\dif t\\ & +(u(T),v(T))- \int_0^T (q,\nabla \cdot u) \dif t= (u_0,v(0)) + \int_0^T ( 
			f, v)
			\dif t.
		\end{split}
	\end{equation}

	\subsection{The numerical method}
	
	To obtain a triangulation of the space-time domain $\Omega \times (0,T)$, we 
	first tessellate the spatial domain $\Omega \subset \mathbb{R}^d$, 
	$d=\cbr{2,3}$ with simplicial elements (if $d=2$), or 
	tetrahedral elements (if  $d=3$). We denote the resulting 
	tessellation by $\mathcal{T}_h = \cbr{K}$, and we assume that it is conforming 
	and quasi-uniform.
	Furthermore, we let $\mathcal{F}_h$ 
	and
	$\partial \mathcal{T}_h$ denote, respectively, the set and union of all edges of
	$\mathcal{T}_h$.  By $h_K$, we denote the 
	diameter of the element $K \in \mathcal{T}_h$, and we let $h = \max_{K \in 
		\mathcal{T}_h} h_{K}$.
	
	Next, we partition the time interval $[0,T]$ into a series of $N+1$ time-levels 
	$0 = 
	t_0 < t_1 < \dots < t_N = T$ of length $\Delta t_n = t_{n+1} - t_n$.
	For simplicity of presentation, we assume a uniform time step size $\Delta t_n 
	= \Delta t$ for $0 \le n \le N$. We remark, however, that a variable 
	time step size poses no additional difficulty in the application nor the 
	analysis of the method. A space-time slab is then defined as 
	$\mathcal{E}^n = 
	\Omega \times I_n$, with $I_n = (t_n,t_{n+1})$. We then tessellate the 
	space-time slab $\mathcal{E}^n$ with space-time prisms $K\times I_n$, i.e.
	$\mathcal{E}^n = \bigcup_{K \in 
		\mathcal{T}_h} K \times I_n$. We denote this tessellation by $\mathscr{T}_h^n$. 
	Combining each space-time slab 
	$n=0,\dots,N-1$, 
	we obtain a tessellation of the space-time domain $\mathscr{T}_h = \bigcup_{n = 
		0}^{N-1} \mathscr{T}_h^n$.

	\subsubsection{The space-time hybridized DG method}
	\label{sss:st-hdg}
	We discretize the Navier--Stokes problem \cref{eq:ns_equations} using the 
	exactly mass conserving hybridized discontinuous Galerkin method developed in 
	\cite{Rhebergen:2018a} combined with a high-order discontinuous Galerkin 
	time stepping scheme. We first introduce the following discontinuous finite 
	element spaces
	on $\mathcal{T}_h$:
	\begin{equation*}
		\begin{split}
			V_h &:= \cbr[1]{v_h \in L^2(\Omega)^d \; | \; v_h \in 
				P_{k_s}(K)^{d} \; \forall K \in \mathcal{T}_h}, \\
			Q_h &:= \cbr[1]{q_h \in L_0^2(\Omega) \; | \; q_h \in 
				P_{k_s-1}(K) \; \forall K \in \mathcal{T}_h}. 
		\end{split}
	\end{equation*}
	On $\partial \mathcal{T}_h$, we introduce the following facet finite element 
	spaces:
	\begin{equation*}
		\begin{split}
			\bar{V}_h &:= \cbr[1]{\bar v_h \in L^2(\partial \mathcal{T}_h) 
				\; | \; 
				\bar v_h\in 
				P_{k_s}(F)^{d} \; \forall F \in \mathcal{F}_h, \;\bar{v}_h|_{\partial \Omega} = 0}, \\
			\bar{Q}_h &:= \cbr[1]{\bar q_h \in L^2(\partial \mathcal{T}_h) \; | \; \bar 
				q_h \in 
				P_{k_s}(F) \; \forall F \in \mathcal{F}_h}.
		\end{split}
	\end{equation*}
	From these spaces, we construct the following space-time finite element spaces 
	in which 
	we will seek our approximation on each space-time slab $\mathcal{E}_h^n$: 
	\begin{equation*}
		\begin{split}
			\mathcal{V}_h &:= \cbr[1]{v_h \in L^2(0,T;L^2(\Omega)^d) \; | 
				\; v_h|_{(t_n,t_{n+1}]}
				\in P_{k_t}(t_n,t_{n+1};V_h)}, \\
			\mathcal{Q}_h &:= \cbr[1]{q_h \in L^2(0,T;L_0^2(\Omega)) \; | \; 
				q_h|_{(t_n,t_{n+1}]} \in P_{k_t}(t_n,t_{n+1};Q_h)}, \\
			\bar{\mathcal{V}}_h &:= \cbr[1]{\bar v_h \in 
				L^2(0,T;L^2(\partial \mathcal{T}_h)^d)
				\; | \; 
				\bar v_h|_{(t_n,t_{n+1}]} \in P_{k_t}(t_n,t_{n+1};\bar{V}_h)}, \\
			\bar{ \mathcal{Q}}_h &:= \cbr[1]{\bar q_h \in L^2(0,T;L^2(\partial 
				\mathcal{T}_h)) \; 
				| \; 
				\bar q_h|_{(t_n,t_{n+1}]} \in P_{k_t}(t_n,t_{n+1};\bar Q_h)}.
		\end{split}
	\end{equation*}
	We note that, in general, the polynomial degree in time $k_t$ can be chosen 
	independently of the polynomial degree in space $k_s$, but for ease of 
	presentation we choose 
	$k_t = k_s = k$. This choice forces us to consider $k_t \ge 1$, but the
	analysis herein is valid also for the case $k_t = 0$ (corresponding to a modified backward Euler
	scheme).
	We adopt the following notation for various product spaces of interest 
	in this work:
	\begin{equation*}
		\boldsymbol{V}_h = V_h \times \bar{V}_h, \quad
		\boldsymbol{Q}_h = Q_h \times \bar Q_h, \quad
		\boldsymbol{\mathcal{V}}_h = \mathcal{V}_h \times \bar{\mathcal{V}}_h,
		\quad  \boldsymbol{\mathcal{Q}}_h = \mathcal{Q}_h \times \bar{\mathcal{Q}}_h.
	\end{equation*}
	Pairs in these product spaces will be denoted using boldface; for example, 
	$\boldsymbol{v}_h := (v_h, \bar{v}_h) \in \boldsymbol{\mathcal{V}}_h$. 
	On each space-time slab $\mathcal{E}^n$, the space-time HDG method for the 
	Navier--Stokes problem reads: find 
	$(\boldsymbol{u}_h, \boldsymbol{p}_h) \in \boldsymbol{\mathcal{V}}_h \times 
	\boldsymbol{\mathcal{Q}}_h$ satisfying for all $(\boldsymbol{v}_h, 
	\boldsymbol{q}_h) \in \boldsymbol{\mathcal{V}}_h \times 
	\boldsymbol{\mathcal{Q}}_h$,
	\begin{subequations} \label{eq:discrete_problem}
		\begin{align}
			-\int_{I_n} &(u_h, \partial_t v_h)_{\mathcal{T}_h} \dif t +
			\label{eq:ns_sthdg_formulation_a}
			\int_{I_n} \del{ \nu a_h(\boldsymbol{u}_h, \boldsymbol{v}_h) + o_h(u_h; 
				\boldsymbol{u}_h, \boldsymbol{v}_h) } \dif t
			\\ \notag
			&	+(u_{n+1}^-,v_{n+1}^-)_{\mathcal{T}_h} +\int_{I_n} b_h(\boldsymbol{p}_h,v_h)  \dif 
			t = (u_{n}^-,v_{n}^+)_{\mathcal{T}_h} + \int_{I_n} (f, 
			v_h)_{\mathcal{T}_h} 
			\dif t, \\
			&\int_{I_n} b_h(\boldsymbol{q}_h,u_h) \dif t = 0,
		\end{align}
	\end{subequations}
	where
	$(u,v)_{\mathcal{T}_h} = \sum_{K \in \mathcal{T}_h}\int_K uv \dif
	x$. We initialize the numerical scheme by choosing $u_0^{-} = P_h u_0$
	on the first space-time slab $\mathcal{E}_h^0$, where
	$P_h : L^2(\Omega) \to V_h^{\text{div}}$ is the $L^2$-projection onto
	$V_h^{\text{div}} := \cbr{u_h \in V_h \; : \;
		b_h(\boldsymbol{q}_h,u_h) = 0,\ \forall \boldsymbol{q}_h \in
		\boldsymbol{Q}_h}$, the discretely divergence free subspace of
	$V_h$.  Here, we denote by $u^{\pm}_n$ the traces at time level $t_n$
	from above and below, i.e.
	$u_n^{\pm} = \lim \limits_{\epsilon \to 0} u_h(t^n \pm \epsilon)$. We
	define the time jump operator at time $t_n$ by
	$\sbr{u_h}_n = u_n^+ - u_n^-$.
	
	The discrete
	multilinear forms $a_h(\cdot,\cdot): \boldsymbol{V}_h \times \boldsymbol{V}_h 
	\to \mathbb{R}$, $b_h(\cdot,\cdot):  \boldsymbol{Q}_h \times V_h
	\to \mathbb{R}$, and $o_h(\cdot;\cdot,\cdot): V_h \times \boldsymbol{V}_h 
	\times \boldsymbol{V}_h \to \mathbb{R}$ appearing in \cref{eq:discrete_problem} 
	serve as approximations to the viscous, pressure-velocity coupling, and 
	convection terms appearing in \cref{eq:continuous_st_problem}, and are defined 
	as:
	\begin{subequations} \label{eq:forms}
		\begin{align}
			\label{eq:formA}
			a_h(\boldsymbol{u}, \boldsymbol{v}) :=&
			\sum_{K\in\mathcal{T}_h}\int_{K} \nabla u : \nabla v \dif x
			+ \sum_{K\in\mathcal{T}_h}\int_{\partial K}\frac{\alpha  }{h_K}(u - 
			\bar{u})\cdot(v - \bar{v}) \dif s,
			\\
			\nonumber
			&- \sum_{K\in\mathcal{T}_h}\int_{\partial 
				K}\sbr{(u-\bar{u})\cdot\partial_n v_h
				+ \partial_nu\cdot(v-\bar{v})} \dif s,
			\\
			\label{eq:formO}
			o_h(w; \boldsymbol{u}, \boldsymbol{v}) :=&
			-\sum_{K\in\mathcal{T}_h}\int_{K} u\otimes w : \nabla v \dif x
			+\sum_{K\in\mathcal{T}_h}\int_{\partial K}\tfrac{1}{2}w\cdot 
			n(u+\bar{u})\cdot(v-\bar{v}) \dif s
			\\
			\nonumber
			&+\sum_{K\in\mathcal{T}_h}\int_{\partial K}\tfrac{1}{2}\envert{w\cdot 
				n}(u-\bar{u})\cdot(v-\bar{v})\dif s,
			\\
			\label{eq:formB}
			b_h(\boldsymbol{p},v) :=&
			- \sum_{K\in\mathcal{T}_h}\int_{K} p \nabla \cdot v \dif x
			+ \sum_{K\in\mathcal{T}_h}\int_{\partial K}v \cdot n \bar{p} \dif s.
		\end{align}
	\end{subequations}
	Here, we slightly abuse notation by using $n$ to denote the outward unit normal $n_K$ to
	the element $K$ for brevity. To ensure stability of the numerical scheme, $\alpha > 0$ must be chosen 
	sufficiently large \cite{Rhebergen:2017}.

	\subsubsection{Preliminaries}
	In this section, we present some preliminaries and rapidly recall the
	main properties of the multilinear forms \cref{eq:forms} that have
	appeared previously in the literature. Throughout this section and the
	rest of the article, we denote by $C > 0$ a generic constant
	independent of the mesh parameters $h$ and $\Delta t$ and the
	viscosity $\nu$, but possibly dependent on the domain $\Omega$, the
	polynomial degree $k$, and the spatial dimension $d$. At times we also
	use the notation $ a \lesssim b$ to denote $a \le C b$. To set
	notation, let
	\begin{subequations}
		\begin{align*}
			V(h) &:= V_h + V \cap H^2(\Omega)^d, 
			\quad                                                                
			\bar{V}(h) := \bar{V}_h + H^{3/2}(\partial 
			\mathcal{T}_h)^d
		\end{align*}
	\end{subequations}
	and define the product space $\boldsymbol{V}(h) := V(h) \times \bar{V}(h)$.
	We introduce the following mesh-dependent inner-products and norms:
	\begin{subequations}
		\begin{align*}
			(\boldsymbol{u}, \boldsymbol{v})_{0,h} &:= (u, 
			v)_{\mathcal{T}_h} 
			+ 	\sum_{K \in \mathcal{T}_h} h_K ( u - \bar{u}, v - 
			\bar{v} )_{\partial K}, && \forall \boldsymbol{u}, 
			\boldsymbol{v} \in \boldsymbol{V}(h),
			\\
			\norm{v}_{1,p,h}^p &:= \sum_{K \in \mathcal{T}_h} \norm{\nabla 
				v}_{L^p(K)}^p 
			+ \sum_{ F \in \mathcal{F}_h} \frac{1}{h^{p-1}_F} 
			\norm{\jump{v}}_{L^p(F)}^p,
			&& \forall v \in V(h), \\
			\tnorm{\boldsymbol{v}}_v^2 &:= \sum_{K\in\mathcal{T}_h} \norm{\nabla 
				v}_{L^2(K)}^2 
			+ \sum_{K\in\mathcal{T}_h}h_K^{-1}\norm{\bar{v} - v}_{L^2(\partial K)}^2, &&\forall \boldsymbol{v}\in \boldsymbol{V}(h),
			\\
			\tnorm{\boldsymbol{v}}_{v'}^2 &:= \tnorm{\boldsymbol{v}}_v^2 
			+ \sum_{K\in\mathcal{T}_h}h_K\norm{(\nabla v)n}_{L^2(\partial K)}^2,
			&& \forall \boldsymbol{v} \in \boldsymbol{V}(h), \\
			\tnorm{\boldsymbol{q}}_p^2 &:= \norm{q_h}_{L^2(\Omega)}^2 + \sum_{K \in \mathcal{T}_h} h_K \norm{\bar{q}_h}_{L^2(\partial K)}^2, && \forall \boldsymbol{q}_h \in \boldsymbol{Q}_h,
		\end{align*}
	\end{subequations}
	where we note that the equivalence constants of $\tnorm{\cdot}_v$ and
	$\tnorm{\cdot}_{v'}$ on the finite-dimensional space
	$\boldsymbol{V}_h$ are independent of the mesh size;
	see~\cite{Rhebergen:2017}. The bilinear form
	$a_h(\cdot, \cdot)$ is continuous and for sufficiently large $\alpha$
	enjoys discrete coercivity \cite[Lemmas 4.2 and 4.3]{Rhebergen:2017},
	i.e. for all $\boldsymbol{v}_h \in\boldsymbol{V}_h$ and
	$\boldsymbol{u}, \boldsymbol{v} \in \boldsymbol{V}(h)$
	\begin{equation}
		\label{eq:ah_coer_bnd}
		a_h(\boldsymbol{v}_h, \boldsymbol{v}_h) \ge  
		C\tnorm{\boldsymbol{v}_h}_v^2 \quad 
		\text{and} \quad
		\envert{a_h(\boldsymbol{u}, \boldsymbol{v})} \le C \tnorm{\boldsymbol{u}}_{v'}\tnorm{\boldsymbol{v}}_{v'}.
	\end{equation}
	%
	%
	%
	The trilinear form $o_h(\cdot;\cdot,\cdot)$ satisfies \cite[Proposition 3.6]{Cesmelioglu:2017}
	\begin{equation}
		\label{eq:ohequals}
		o_h(w_h; \boldsymbol{v}_h, \boldsymbol{v}_h) = 
		\frac{1}{2}\sum_{K\in\mathcal{T}}\int_{\partial K}\envert{w_h\cdot 
			n}\envert{v_h - \bar{v}_h}^2 \dif s \ge 0 \qquad w_h \in V_h^{\text{div}}, \;
		\forall \boldsymbol{v}_h \in\boldsymbol{V}_h.
	\end{equation}
	Further, the trilinear form $o_h(\cdot;\cdot,\cdot)$ is Lipschitz continuous in 
	its 
	first argument~\cite[Proposition 3.4]{Cesmelioglu:2017}: for all
	$w_1, w_2 \in V(h)$, $\boldsymbol{u} \in 
	\boldsymbol{V}(h)$ and
	$\boldsymbol{v} \in \boldsymbol{V}(h)$ it holds that
	\begin{equation}
		\label{eq:ohLip}
		\envert{o_h(w_{1}; \boldsymbol{u}, \boldsymbol{v}) - o_h(w_{2}; 
			\boldsymbol{u}, \boldsymbol{v})} \le 
		C\norm{w_{1}-w_{2}}_{1,h}\tnorm{ \boldsymbol{u} }_{v}\tnorm{ 
			\boldsymbol{v} }_{v}.
	\end{equation}
	
	\subsection{Well-posedness and stability}
	
	To the best of the authors' knowledge, a rigourous study of well-posedness for 
	higher-order space-time Galerkin schemes applied to the Navier--Stokes 
	equations 
	has yet to appear in the literature. We remark that for the lowest order scheme
	($k = 1)$, uniqueness of the discrete solution is a consequence of the following
	energy estimate which we will derive in \Cref{sec:well_posedness}:
	\begin{lem} \label{lem:energy_stability}
		Let $d=2 \text{ or } 3$, $k \ge 1$, and suppose that $\boldsymbol{u}_h \in 
		\boldsymbol{\mathcal{V}}_h$ is an approximate 
		velocity 
		solution 
		of the Navier--Stokes equations computed using the space-time HDG scheme 
		\cref{eq:discrete_problem} for $n=0,\dots,N-1$. There exists a 
		$C>0$ such that
		\begin{equation*}
			\norm[1]{u_{N}^-}_{L^2(\Omega)}^2 
			+ \sum_{n=0}^{N-1} \norm[1]{ 
				\sbr{u_h}_n}_{L^2(\Omega)}^2 +
			\nu \int_{0}^{T}\tnorm{\boldsymbol{u}_h}_v^2 \dif t  \le C\del{
				\frac{1}{\nu}\int_{0}^{T}\norm{f}_{L^2(\Omega)}^2 \dif t \; + \;
				\norm[1]{u_{0}}_{L^2(\Omega)}^2} .
		\end{equation*}
	\end{lem}
	
	For higher order schemes in time $(k \ge 2)$, this energy bound is
	insufficient to prove the uniqueness of the discrete solution as
	$\text{ess\;sup}_{0 < t \le T} \; u_h(t)$ need not be attained at the
	partition points of the time-interval. Consequently,
	\Cref{lem:energy_stability} offers no uniform control over the
	discrete velocity solution $u_h$ in $L^{\infty}(0,T;L^2(\Omega)^d)$.
	We will show that such a bound is possible if $d=2$ using the tools
	introduced by Chrysafinos and Walkington
	\cite{Chrysafinos:2006,Chrysafinos:2008,Chrysafinos:2010}.
	
	We begin by introducing the exponential interpolant from
	\cite{Chrysafinos:2010}. Let $(V, (\cdot, \cdot)_V)$ be an
	inner-product space and let $\lambda > 0$ be given. The exponential
	interpolant $\tilde{v}$ of $v \in P_k(I_n;V)$ is defined by
	\begin{equation}
		\label{eq:exp_interp_property}
		\int_{I_n} \del{\tilde{v}, q}_V \dif t = \int_{I_n} \del{v, q}_V e^{-\lambda(t-t_n)} \dif t, \quad \forall q \in P_{k-1}(I_n, V). 
	\end{equation}
	such that $\tilde{v}(t_n^+) = v(t_n^+)$. By \cite[Lemma 3.4 and Lemma
	3.6]{Chrysafinos:2010} there exists a constant $C>0$ such that for all
	$v \in P_k(I_n;V)$ and $1 \le p \le \infty$,
	\begin{equation} \label{eq:lp_bnd_exp_interp}
		\norm{\tilde{v}}_{L^p(I_n;V)} \le C  \norm{v}_{L^p(I_n;V)}.
	\end{equation}
	
	Next, the discrete characteristic function of $v \in P_k(I_n;V)$ for
	fixed $s \in (t_n,t_{n+1})$ is defined as the function
	$v_{\chi} \in P_k(I_n;V)$ satisfying $v_{\chi}(t_n^+) = v(t_n^+)$ and
	\begin{equation}
		\label{eq:disc_char_function}
		\int_{t_n}^{t_{n+1}}(v_{\chi}, w)_V \dif t= \int_{t_n}^s (v, w)_V \dif t, \quad 
		\forall 
		w \in P_{k-1}(I_n;V),
	\end{equation}
	(see also \cite{Dolejsi:book}). By \cite[Lemmas 3.1 and
	3.2]{Chrysafinos:2010} the discrete characteristic function satisfies
	\begin{equation} \label{eq:bnd_char}
		\norm[0]{v_{\chi}}_{L^2(I_n,V)} \le C \norm[0]{v}_{L^2(I_n,V)}.
	\end{equation}
	Moreover, if $v(t) = z$ is constant in time,  
	its 
	discrete characteristic function can be characterized by $v_\chi = p(t)z$ for $p \in P_k(I_n)$ 
	satisfying 
	$p(t_n^+) = 1$ and
	\begin{subequations}
		\begin{align}
			\label{eq:const_int_char}
			\int_{t_n}^{t_{n+1}} p q \dif t &= \int_{t_n}^s q \dif t, \quad \forall 
			q \in P_{k-1}(I_n), \\
			\label{eq:const_linf_bnd_char}
			\norm[0]{p}_{L^{\infty}(I_n)} &\le C.
		\end{align}
	\end{subequations}
	With the help of these tools, it is possible to bound the discrete solution 
	$u_h$ in 
	$L^{\infty}(0,T;L^2(\Omega)^d)$ in two spatial dimensions:
	\begin{lem} \label{lem:linfty_bnd}
		Let $d=2$, $k \ge 1$, and suppose $\boldsymbol{u}_h \in 
		\boldsymbol{\mathcal{V}}_h$ is an 
		approximate 
		velocity 
		solution 
		of the Navier--Stokes equations computed using the space-time HDG scheme 
		\cref{eq:discrete_problem} for $n=0,\dots,N-1$. There exists a 
		constant $C>0$ such that
		\begin{equation*} 
			\begin{split}
				&
				\norm{u_h}_{L^{\infty}(0,T;L^2(\Omega))}^2
				\\
				&\le C\del[3]{ \frac{1}{\nu}\int_{0}^{T}\norm{f}_{L^2(\Omega)}^2 
					\dif 
					t \; + \;
					\norm[1]{u_{0}}_{L^2(\Omega)}^2} +
				\frac{C}{\nu^2}\del[3]{ 
					\frac{1}{\nu}\int_{0}^{T}\norm{f}_{L^2(\Omega)}^2 
					\dif t \; + \;
					\norm[1]{u_{0}}_{L^2(\Omega)}^2}^2.
			\end{split}
		\end{equation*}
	\end{lem}
	\noindent With this bound in hand, we 
	can prove the following uniqueness result in two dimensions for the 
	solution of the nonlinear 
	system of algebraic equations arising from the discrete scheme 
	\cref{eq:discrete_problem}:
	\begin{thm}[Uniqueness in two dimensions] 
		\label{thm:uniqueness_HO}
		\label{thm:unique2d}
		Let $\boldsymbol{u}_h \in \boldsymbol{\mathcal{V}}_h$ be an
		approximate velocity solution of the Navier--Stokes equations
		computed using the space-time HDG scheme
		\cref{eq:discrete_problem} for $n=0,\dots,N-1$. In two
		dimensions, if the problem data satisfies
		\cref{eq:small_data_2}
		then $\boldsymbol{u}_h$ is the unique velocity solution to 
		\cref{eq:discrete_problem}.
	\end{thm}
	We defer the proofs of \Cref{lem:linfty_bnd} and \Cref{thm:uniqueness_HO} to 
	\Cref{sec:well_posedness}. In addition to the bound on $u_h$ in 
	$L^{\infty}(0,T;L^2(\Omega)^d)$, the other key ingredient for proving 
	\Cref{thm:unique2d} is a novel \emph{discrete 
		version} of 
	the classic Ladyzhenskaya inequalities (see e.g. \cite[Section II.3]{Galdi:book})
	valid 
	for \emph{broken polynomial spaces}. We will discuss this further in 
	\Cref{sec:preliminaries}. Note that, similar to the continuous theory, the 
	scaling of the exponents in the discrete Ladyzhenskaya inequality with respect to the spatial dimension prevents us from extending the proof of uniqueness to $d=3$.
	
	\subsection{Error analysis}

	Our main result is a \emph{pressure-robust} 
	error estimate for the approximate 
	velocity arising from the numerical scheme \cref{eq:discrete_problem} under the 
	assumption that the problem data satisfies  
	\cref{eq:small_data_2}:%
	\begin{thm}[Velocity error] \label{thm:vel_error_estimates}
		Let $u$ be the strong velocity solution to the Navier--Stokes system
		\cref{eq:ns_equations} guaranteed by \Cref{thm:strong_solution} and assume
		it further satisfies
		\begin{equation*}
			u \in H^{k+1}(0,T;V\cap 
			H^2(\Omega)^d) \cap H^1(0,T;H^{k+1}(\Omega)^d),
		\end{equation*}
		with initial data $u_0 \in 
		H^{k+1}(\Omega)^d$. 
		Let $(u_h,\bar u_h) \in 
		\boldsymbol{\mathcal{V}}_h $ be an 
		approximate velocity solution 
		to the Navier--Stokes system computed using the space-time HDG scheme
		\cref{eq:discrete_problem} for $n= 0, \dots, N-1$ ,
		Then, there exists a constant $C>0$ such that the error $\boldsymbol{e}_h = 
		(u-u_h,\gamma(u) - \bar{u}_h)$ satisfies
		\begin{equation*}
			\norm[1]{e_{N}^-}_{L^2(\Omega)}^2   + \sum_{n=0}^{N-1} 
			\norm{\sbr{e_h}_n}_{L^2(\Omega)}^2 + 
			\nu \int_{0}^{T} \tnorm{ \boldsymbol{e}_h}_{v'}^2 \dif t 
			\le  \exp 
			\del[2]{CT} \del{   h^{2k} + \Delta t^{2k+2}  } C(u),
		\end{equation*}
		provided the time step satisfies $\Delta t \lesssim \nu$. Here, $C(u)$ depends on 
		Sobolev--Bochner norms of the velocity $u$, but is 
		independent of the pressure $p$.
	\end{thm}
	The proof of \Cref{thm:vel_error_estimates} is deferred to
	\Cref{sec:vel_err}. We remark that the time step restriction
	$\Delta t \lesssim \nu$ in \Cref{thm:vel_error_estimates} is necessary
	in the proof of this theorem to use a discrete Gr\"onwall inequality;
	it is not necessary for the stability of the space-time HDG method
	\cref{eq:discrete_problem}, but rather to \emph{quantify the
		asymptotic rates of convergence}.

	\section{Preliminary results} \label{sec:preliminaries}
	
	\subsection{Properties of the numerical scheme}
	
	Let $\mathcal{V}_h^{\text{div}}$ denote the
	subspace of $\mathcal{V}_h$ of discrete divergence free velocity
	fields:
	\begin{equation*}
		\mathcal{V}_h^{\text{div}} = \cbr[2]{ u_h \in \mathcal{V}_h \; : \;
			\int_{I_n} b_h(\boldsymbol{q}_h,u_h)\dif t = 0, \quad \forall 
			\boldsymbol{q}_h \in \boldsymbol{\mathcal{Q}}_h}.
	\end{equation*}
	The following result motivates the 
	use of equal order polynomial degrees in time for both the velocity and 
	pressure approximation spaces:
	\begin{lem} \label{lem:kernel_equiv}
		$\mathcal{V}_h^{\text{div}} = \cbr{v_h \in \mathcal{V}_h \; | \;  v_h|_{\mathcal{E}^n} \in P_k(I_n;V_h^{\text{div}})}$.
	\end{lem}
	\begin{proof}
		The proof is very similar to that of \cite[Lemma 
		2.3]{Chrysafinos:2010} with minor modifications and is therefore omitted. 
	\end{proof}
	An immediate consequence of \Cref{lem:kernel_equiv} is that
	$u_h(t) \in H$ a.e. $t \in (0,T)$ where $H$ is defined in
	\cref{eq:defH}. To see this we first expand $u_h$ in terms of an
	orthonormal basis $\cbr[0]{\phi_i}_{i=0}^k$ of $P_k(I_n)$ with respect
	to the $L^2(I_n)$ inner-product:
	\begin{equation}
		\label{eq:expansionuh}
		u_h = \sum_{i=0}^k \phi_i(t) u_i(x), \quad u_i \in V_h.
	\end{equation} 
	By \Cref{lem:kernel_equiv}, $u_h \in P_k(I_n;V_h^{\text{div}})$, so
	$u_i \in V_h^{\text{div}}$ for each $i=0,\dots,k$. Thus,
	\begin{equation*}
		0 = b_h(\boldsymbol{q}_h,u_i) = - \sum_{K\in\mathcal{T}_h}\int_{K}q_h 
		\nabla 
		\cdot u_i \dif x
		+ \sum_{K\in\mathcal{T}_h}\int_{\partial K} u_i \cdot n \bar{q}_h \dif 
		s, \quad 
		\forall \boldsymbol{q}_h \in \boldsymbol{Q}_h.
	\end{equation*}
	Following the same arguments as \cite[Proposition 1]{Rhebergen:2018a}
	it follows that $\nabla \cdot u_i = 0$ for all $x \in K$,
	$\jump{u_i \cdot n} = 0$ on all $F \in \mathcal{F}_h^{\text{int}}$,
	and $u_i \cdot n = 0$ on $\partial \Omega$ for $i=0,\hdots, k$. By
	\cref{eq:expansionuh} and since $H$ is a linear space the result
	follows.
	
	\begin{lem}[Consistency] \label{lem:consistency}
		Let $(u,p)$ be the strong solution to the Navier--Stokes system
		\cref{eq:ns_equations} guaranteed by \Cref{thm:strong_solution}.
		Define $\boldsymbol{u} = (u,\gamma(u))$ and
		$\boldsymbol{p} = (p,\gamma(p))$. Then, it holds that
		\begin{multline*}
			-\int_{I_n} (u, \partial_t v_h)_{\mathcal{T}_h} \dif t + 
			(u(t_{n+1}),v_{n+1}^-)_{\mathcal{T}_h} +  
			\int_{I_n} ( \nu a_h(\boldsymbol{u}, \boldsymbol{v}_h) + o_h(u; 
			\boldsymbol{u}, \boldsymbol{v}_h) + b_h(\boldsymbol{p}, v_h) ) \dif t                
			\\
			- \int_{I_n} b_h(\boldsymbol{q}_h, u) \dif t
			= (u(t_n),v_{n}^+)_{\mathcal{T}_h} + \int_{I_n} (f, 
			v_h)_{\mathcal{T}_h} 
			\dif t, \quad \forall (\boldsymbol{v}_h,\boldsymbol{q}_h) \in \boldsymbol{\mathcal{V}}_h \times \boldsymbol{\mathcal{Q}}_h.
		\end{multline*}
	\end{lem}
	
	\subsection{Scalings and embeddings}
	\label{ss:scale_embed}
	
	We begin by recalling a number of results for piece-wise
	polynomials. First, for polynomials in time, let
	$(V, (\cdot, \cdot)_V)$ be an inner product space. Then, there exists
	$C > 0$ such that for all $v \in P_k(I_n, V)$ (see e.g. \cite[Lemma
	3.5]{Chrysafinos:2010}):
	\begin{subequations}
		\begin{align}
			\label{eq:time_scaling}
			\norm{v}_{L^p(I_n,V)} &\le C \Delta t^{1/p-1/2} \norm{v}_{L^2(I_n,V)}, \qquad 1\le p \le \infty,
			\\
			\label{eq:time_inverse}
			\norm{\partial_t v}_{L^2(I_n,V)} &\le C \Delta t^{-1}\norm{v}_{L^2(I_n,V)}.
		\end{align}
	\end{subequations}
	
	\noindent Next, we recall the following discrete version of the
	Sobolev embedding theorem valid for broken polynomial spaces
	$P_r(\mathcal{T}_h) = \cbr{ f \in L^2(\Omega) \; | \; f|_{K} \in
		P_{r}(K), \; \forall K \in \mathcal{T}_h}$ where $r \ge 0$. Let
	$ 1\le p < \infty$, then for all $q$ satisfying
	$1 \le q \le pd/(d-p) \text{ if } 1 \le p < d$, or
	$1 \le q < \infty \text{ if } d \le p < \infty$, there exists a
	constant $C>0$ such that~\cite[Theorem 5.3]{Pietro:book}:
	\begin{equation} \label{eq:disc_sobolev}
		\norm{v_h}_{L^q(\Omega)} \le 
		C \norm{v_h}_{1,p,h},  \quad \forall v_h \in P_r(\mathcal{T}_h).
	\end{equation}
	In the case $p=2$, we write
	$\norm{\cdot}_{1,2,h} = \norm{\cdot}_{1,h}$. Note that choosing
	$p=q=2$ in \cref{eq:disc_sobolev} yields the discrete Poincar\'{e}
	inequality: $\norm{v_h}_{L^2(\Omega)} \le C_P \norm{v_h}_{1,h}$ for
	all $v_h \in V_h$. By the triangle inequality,
	$\norm{v_h}_{1,h} \le \tnorm{\boldsymbol{v}_h}_v$, so that
	\begin{equation}
		\label{eq:disc_poinc}
		\norm{v_h}_{L^2(\Omega)} \le C_P \tnorm{\boldsymbol{v}_h}_{\nu}, \quad \forall \boldsymbol{v}_h \in \boldsymbol{V}_h.
	\end{equation}

	We now prove a discrete version of the Ladyzhenskaya inequalities
	valid for broken polynomial spaces. While the analogue of these
	inequalities are well known in the context of $H^1$-conforming finite
	element methods \cite{Dupont:2009}, to our knowledge they have yet to
	be extended to \emph{non-conforming} finite element spaces.
	
	\begin{lem}[Ladyzhenskaya inequalities for broken polynomial spaces]
		There exists a constant $C>0$ such that for $d\in\cbr{2,3}$:
		\begin{equation} \label{eq:disc_lady_2d}
			\norm{v_h}_{L^4(\Omega)} \le C \norm{v_h}^{1/2(d-1)}_{L^2(\Omega)}
			\norm{v_h}_{1,h}^{3/2(5-d)}, \quad \forall v_h \in V_h.
		\end{equation}	
	\end{lem}
	\begin{proof}
		It suffices to consider the scalar case.  We focus first on
		the case $d=2$. Inserting $v_h^2$
		into~\cref{eq:disc_sobolev} with $r = 2k$, $p=1$, and
		$q = 2$ yields
		$\norm{v_h}_{L^4(\Omega)}^2 \le C
		\norm[1]{v_h^2}_{1,1,h}$. The result follows after noting
		that the right hand side can be bounded by applying the
		Cauchy--Schwarz inequality and a local discrete trace
		inequality
		$\norm{v_h}_{{L^2(F)}} \le C h_K^{-1/2} \norm{v_h}_{L^2(K)}$
		\cite[Lemma 1.46]{Pietro:book}:
		\begin{equation*}
			\label{eq:lady_inter_2}
			\frac{1}{2}\norm[1]{v_h^2}_{1,1,h}
			=   \sum_{K \in \mathcal{T}_h} \int_{K} \left|(\nabla v_h) 
			v_h \right| 
			\dif x
			+ \sum_{ F \in \mathcal{F}_h} \int_F  \left| \jump{v_h}\cdot 
			\av{v_h} \right| \dif s 
			\le C \norm{v_h}_{L^2(\Omega)} 
			\norm{v_h}_{1,h}.
		\end{equation*}
		For the case $d=3$, the result follows from the
		Cauchy--Schwarz inequality and \cref{eq:disc_sobolev} with
		$q = 6$ and $p=2$.
	\end{proof}
	
	For $d=3$, interpolating between $L^2(\Omega)^d$ and $L^4(\Omega)^d$ and using  \cref{eq:disc_lady_2d} yields:
	\begin{equation} \label{eq:L3_bnd_3d}
		\norm{v_h}_{L^3(\Omega)} \le C \norm{v_h}^{1/2}_{L^2(\Omega)}
		\norm{v_h}_{1,h}^{1/2}, \quad \forall v_h \in V_h.	
	\end{equation}

	\section{Well-posedness of the discrete problem} \label{sec:well_posedness}
	\subsection{Existence of a discrete solution}
	
	We will begin by showing the existence of a solution to the nonlinear
	system of algebraic equations arising from \cref{eq:discrete_problem}
	by making use of the following topological degree result taken from
	\cite[Lemma 6.42]{Pietro:book}:
	\begin{lem}\label{lem:top_degree}
		Let $(X , \; \norm{\cdot}_X)$ be a finite-dimensional normed space. Let $M 
		> 0$ 
		and let $\Psi : X \times [0,1] \to X$ satisfy
		\begin{enumerate}
			\item $\Psi$ is continuous.
			\item $\Psi(\cdot,0)$ is an affine function and the equation $\Psi(v,0) 
			= 0$
			has a solution $v \in X$ such that $\norm{v}_X < M$.
			\item For any $(v,\rho) \in X \times [0,1]$, $\Psi(v,\rho) = 0$ implies
			$\norm{v}_X < M$.
		\end{enumerate}
		Then, there exists $v \in X$ such that $\Psi(v,1) = 0$ and $\norm{v}_X < M$.
	\end{lem}
	\noindent To apply \Cref{lem:top_degree} in \cref{sss:toplogdegarg} we
	first require the proof of \Cref{lem:energy_stability} and
	well-posedness of the space-time HDG discretization of the linear
	time-dependent Stokes problem as discussed next.
	
	\subsubsection{Proof of \Cref{lem:energy_stability}.}
	\begin{proof}
		Testing \cref{eq:discrete_problem} with $(v_h,\bar{v}_h,q_h,\bar{q}_h) 
		= (u_h,\bar{u}_h,p_h,\bar{p}_h)$,
		using the coercivity of $a_h(\cdot,\cdot)$ and the fact that 
		$o_h(\cdot; 
		\boldsymbol{u}_h, \boldsymbol{u}_h) \ge 0$, and integrating by parts in time,
		we find that there exists a constant $C_1>0$ such that
		\begin{equation*}
			\begin{split}
				\frac{1}{2} \norm[1]{u_{n+1}^-}_{L^2(\Omega)}^2  +& \; 
				\frac{1}{2}\norm[1]{ 
					\sbr{u_h}_n}_{L^2(\Omega)}^2 - 
				\frac{1}{2}\norm[1]{u_{n}^-}_{L^2(\Omega)}^2 \;+ \;
				C_1 \nu \int_{I_n} \tnorm{\boldsymbol{u}_h}_v^2 \dif t \le \int_{I_n} 
				(f , 
				u_h)_{\mathcal{T}_h} \dif t . 
			\end{split}
		\end{equation*}
		To bound the right hand side, we apply the Cauchy--Schwarz inequality, 
		the discrete Poincar\'{e} inequality \cref{eq:disc_poinc}, and Young's
		inequality with $\epsilon = C_1/(C_P\nu) > 0$. Rearranging, we see there is a constant $C_2>0$ 
		such that
		\begin{equation*}
			\norm[1]{u_{n+1}^-}_{L^2(\Omega)}^2 
			+ \norm[1]{ 
				\sbr{u_h}_n}_{L^2(\Omega)}^2 - 
			\norm[1]{u_{n}^-}_{L^2(\Omega)}^2+
			\nu C_1\int_{I_n} \tnorm{\boldsymbol{u}_h}_v^2 \dif t\le \frac{C_2}{ 
				\nu}\int_{I_n} \norm{f}_{L^2(\Omega)}^2 \dif t.
		\end{equation*}
		The result follows after summing over all space-time slabs. 
	\end{proof}
	
	\subsubsection{A linearized problem}
	Before we can apply the topological degree argument, we will need to study the 
	space-time HDG solution of the linear time-dependent Stokes problem: 
	\begin{subequations} \label{eq:stokes_equations}
		\begin{align}
			\partial_t u - \nu \Delta u + \nabla p &= f, && 
			\hspace{-10mm}
			\text{in } \Omega\times (0,T], \\
			\nabla \cdot u &= 0, && \hspace{-10mm}\text{in } \Omega\times (0,T], \\
			u &=0, && \hspace{-10mm}\text{on }  \partial \Omega\times (0,T], \\
			u(x,0) &= u_0(x), && \hspace{-10mm}\text{in } \Omega.
		\end{align}
	\end{subequations}
	\begin{lem} \label{lem:wellposed_stokes}
		There exists a unique pair $\boldsymbol{u}_h \in \mathcal{V}_h^{\text{div}} 
		\times \bar{\mathcal{V}}_h$ such that for all $\boldsymbol{v}_h \in \mathcal{V}_h^{\text{div}} 
		\times \bar{\mathcal{V}}_h$:
		\begin{equation}\label{eq:discrete_stokes_problem}
			-\int_{I_n} (u_h, \partial_t v_h)_{\mathcal{T}_h} \dif t + 
			(u_{n+1}^-,v_{n+1}^-)_{\mathcal{T}_h} +
			\int_{I_n} \nu a_h(\boldsymbol{u}_h, \boldsymbol{v}_h) \dif t \\
			= (u_{n}^-,v_{n}^+)_{\mathcal{T}_h} + \int_{I_n} (f, 
			v_h)_{\mathcal{T}_h} 
			\dif t.
		\end{equation}
		Note that this is simply the space-time HDG scheme applied to the 
		time-dependent Stokes problem \cref{eq:stokes_equations}.
	\end{lem}
	\begin{proof}
		The result follows from the	Lax--Milgram theorem.
	\end{proof}

	\subsubsection{The topological degree argument}
	\label{sss:toplogdegarg}
	
	\begin{thm} \label{thm:existence} Let $d\in \cbr{2,3}$ and $k\ge 1$.
		There exists at least one discrete velocity solution $\boldsymbol{u}_h \in 
		\mathcal{V}_h^{\text{div}} 
		\times \bar{\mathcal{V}}_h$ to \cref{eq:discrete_problem} for $n=0,\dots,N-1$ 
		satisfying the energy 
		estimate 
		\Cref{lem:energy_stability}.
	\end{thm}
	\begin{proof}
		We set $X = \mathcal{V}_h^{\text{div}} \times \bar{\mathcal{V}}_h$ and 
		equip it with the norm
		\begin{equation*}
			\norm{\boldsymbol{u}_h}_X^2 := \norm[1]{u^{-}_{N}}_{L^2(\Omega)}^2 	
			+ \sum_{n=0}^{N-1} \norm[1]{ 
				\sbr{u_h}_n}_{L^2(\Omega)}^2
			+ \nu \int_{0}^T\tnorm{\boldsymbol{u}_h}_{v}^2 \dif t.
		\end{equation*}
		Define the continuous mapping $\Psi: X \times [0,1] \to X$ for $n=0,\dots, N-1,$
		by
		\begin{equation*}
			\begin{split}
				\int_{I_n}	(\Psi(\boldsymbol{u}_h, \rho), \boldsymbol{v}_h)_{0,h} \dif t=&  
				-\int_{I_n} 
				(u_h, \partial_t v_h)_{\mathcal{T}_h} \dif t + 
				(u_{n+1}^-,v_{n+1}^-)_{\mathcal{T}_h} +
				\int_{I_n}  \nu a_h(\boldsymbol{u}_h, \boldsymbol{v}_h) \dif t  \\
				& + \int_{I_n} \rho 
				o_h(u_h; 
				\boldsymbol{u}_h, \boldsymbol{v}_h)  \dif t- (u_{n}^-,v_{n}^+)_{\mathcal{T}_h} - \int_{I_n} (f, 
				v_h)_{\mathcal{T}_h} 
				\dif t. 
			\end{split}
		\end{equation*}
		$\Psi$ is well-defined by the Riesz 
		representation theorem, verifying item (1) in~\Cref{lem:top_degree}. 
		Next, we 
		choose $\boldsymbol{u}_h \in X$ such that
		$\Psi(\boldsymbol{u}_h,\rho) = 0$ for some $\rho \in [0,1]$. 
		Since $o_h(u_h;\boldsymbol{u}_h,\boldsymbol{u}_h) \ge 0$, we can repeat
		the proof of \Cref{lem:energy_stability} to bound 
		$\boldsymbol{u}_h$ 
		uniformly with respect to $\rho$:
		\begin{equation*}
			\norm[1]{u_{N}^-}_{L^2(\Omega)}^2 
			+ \sum_{n=0}^{N-1} \norm[1]{ 
				\sbr{u_h}_n}_{L^2(\Omega)}^2 +
			\nu \int_{0}^{T}\tnorm{\boldsymbol{u}_h}_v^2 \dif t \le C\del{
				\frac{1}{\nu}\int_{0}^{T}\norm{f}_{L^2(\Omega)}^2 \dif t \; + \;
				\norm[1]{u_{0}}_{L^2(\Omega)}^2} ,
		\end{equation*}
		which verifies item (3) in~\Cref{lem:top_degree} with
		\begin{equation*}
			M^2 = C\del{
				\frac{1}{\nu}\int_{0}^{T}\norm{f}_{L^2(\Omega)}^2 \dif t \; + \;
				\norm[1]{u_{0}}_{L^2(\Omega)}^2} + \epsilon,
		\end{equation*}
		for any $\epsilon > 0$.
		Finally, note that $\Psi(\cdot,0): X \to X$ is an affine function since the 
		nonlinear convection term disappears for $\rho = 0$. By 
		\Cref{lem:wellposed_stokes}, there exists a solution to 
		$\Psi(\boldsymbol{u}_h,0) 
		= 0$, verifying item (2) in~\Cref{lem:top_degree}. Therefore, there exists
		a solution $\boldsymbol{u}_h$ to $\Psi(\boldsymbol{u}_h,1) = 0$
		satisfying $\norm{\boldsymbol{u}_h}_X < M$.
		Equivalently, $\boldsymbol{u}_h \in 
		\mathcal{V}_h^{\text{div}} \times 
		\bar{\mathcal{V}}_h$
		solves \cref{eq:discrete_problem} for all $\boldsymbol{v}_h \in 
		\mathcal{V}_h^{\text{div}} \times 
		\bar{\mathcal{V}}_h$ and satisfies the energy bound in \Cref{lem:energy_stability}.
	\end{proof}
	\subsection{Uniqueness of the discrete velocity solution in two dimensions}
	\subsubsection{Bounds on the trilinear convection term}
	In the analysis that follows, we will require tighter bounds on the trilinear 
	convection form than is provided by \cref{eq:ohLip}. For this, we will make 
	extensive 
	use of the results of \Cref{ss:scale_embed}. We remark that, although we focus 
	on $d=2$ for the proof of uniqueness, the bound \cref{eq:o_space_bnd_2} will be 
	essential for the error analysis in both two and three spatial dimensions in 
	\Cref{sec:vel_err}.

	\begin{lem} \label{lem:o_space_bnd} If $d = 2$, there exists a $C>0$
		such that for all
		$\boldsymbol{w}_h, \boldsymbol{u}_h,\boldsymbol{v}_h \in
		\boldsymbol{V}_h$, 
		\begin{equation} \label{eq:o_space_bnd_1}
			\begin{split}
				&|o_h(w_h;\boldsymbol{u}_h,\boldsymbol{v}_h)|  \\ 
				&\le 
				C  \norm{w_h}_{L^2(\Omega)}^{1/2}\tnorm{\boldsymbol{u}_h}^{1/2}_{v} 
				\tnorm{\boldsymbol{v}_h}_{v}
				\del{
					\norm{w_h}_{L^2(\Omega)}^{1/2}\tnorm{\boldsymbol{u}_h}_{v}^{1/2}
					+ \norm{u_h}_{L^2(\Omega)}^{1/2} 
					\tnorm{\boldsymbol{w}_h}^{1/2}_{v} }.
			\end{split}
		\end{equation}
		Moreover, if $d \in \cbr{2,3}$, there exists a $C>0$
		such that
		for all $\boldsymbol{w}_h, \boldsymbol{u}_h,\boldsymbol{v}_h \in 
		\boldsymbol{V}_h$, 
		\begin{equation} \label{eq:o_space_bnd_2}
			|o_h(w_h;\boldsymbol{u}_h,\boldsymbol{v}_h)| \le C 
			\norm{w_h}_{L^2(\Omega)}^{1/2} \tnorm{\boldsymbol{w}_h}_{v}^{1/2} 
			\tnorm{\boldsymbol{u}_h}_v \tnorm{\boldsymbol{v}_h}_v.
		\end{equation}
	\end{lem}
	\begin{proof}
		This proof relies on the following scaling identity for
		$\mu \in P_k(\partial K)$ between $L^p$ and $L^2$ norms on element
		boundaries which can be obtained using standard arguments:
		\begin{equation}
			\label{lem:facet_scaling}
			\norm{\mu}_{L^p(\partial K)} \le Ch^{(d-1)(1/p-1/2)}
			\norm{\mu}_{L^2(\partial K)}, \qquad 2 \le p < \infty.
		\end{equation}
		Now, split $o_h(w_h,\boldsymbol{u}_h,\boldsymbol{v}_h)$ 
		into three terms and bound each separately. Using that $u_h + \bar{u}_h = 
		2u_h 
		+ (\bar{u}_h - u_h)$, we find:
		\begin{equation}
			\begin{split}
				|o_h&(w_h,\boldsymbol{u}_h,\boldsymbol{v}_h)| \\
				\le&\sum_{K \in \mathcal{T}_h} 
				\int_{K} |(u_h\otimes w_h) : \nabla v_h| \dif x
				+ \sum_{K \in \mathcal{T}_h} \int_{\partial K}| w_h 
				\cdot n 
				\del{u_h} \cdot 
				(v_h-\bar{v}_h)| \dif s \\
				&+ \sum_{K \in \mathcal{T}_h} \int_{\partial K}|w_h 
				\cdot 
				n||(u_h-\bar{u}_h)  \cdot 
				(v_h-\bar{v}_h)| \dif s = T_1 + T_2 + T_3.
			\end{split}
		\end{equation}
		To show \cref{eq:o_space_bnd_1}, we first apply the generalized H\"{o}lder 
		inequality to $T_1$ with $p = q = 4$ and $r=2$, the Cauchy--Schwarz 
		inequality, and \cref{eq:disc_lady_2d} to find:
		\begin{equation*}
			\begin{split}
				|T_1| 
				& \le   C \norm{u_h}_{L^2(\Omega)}^{1/2} 
				\tnorm{\boldsymbol{u}_h}_v^{1/2} \norm{w_h}_{L^2(\Omega)}^{1/2} 
				\tnorm{\boldsymbol{w}_h}_v^{1/2} \tnorm{\boldsymbol{v}_h}_v. 
			\end{split}
		\end{equation*}
		To bound $T_2$, we apply the generalized H\"{o}lder inequality 
		with $p=q=4$ and $r=2$ and use the local discrete trace inequality 
		$\norm{v_h}_{L^p(\partial 
			K)} \le C h_K^{-1/p}\norm{v_h}_{L^p(K)}$ (see e.g. \cite[Lemma 
		1.52]{Pietro:book}) for $p=4$, we have
		\begin{equation*}
			\begin{split}
				|T_2| &\le C \sum_{K \in \mathcal{T}_h} 
				\norm{w_h}_{L^4(K)}  \norm{u_h}_{L^4(K)} 
				h_K^{-1/2}\norm{v_h - \bar{v}_h}_{L^2(\partial K)}.
			\end{split}
		\end{equation*}
		Applying the Cauchy--Schwarz inequality and \cref{eq:disc_lady_2d} we find
		\begin{equation*} \label{eq:T_2_inter}
			\begin{split}
				|T_2| 
				& \le C \norm{u_h}_{L^2(\Omega)}^{1/2} 
				\tnorm{\boldsymbol{u}_h}_v^{1/2} \norm{w_h}_{L^2(\Omega)}^{1/2} 
				\tnorm{\boldsymbol{w}_h}_v^{1/2} \tnorm{\boldsymbol{v}_h}_v. 
			\end{split}
		\end{equation*}
		To bound $T_3$, we again apply the generalized H\"{o}lder inequality 
		with $p=q=4$ and $r=2$, the local discrete trace inequality $\norm{v_h 
			\cdot 
			n}_{L^2(\partial K)} \le C h_K^{-1/2} \norm{v_h}_{L^2(K)}$,  
		\cref{lem:facet_scaling} with $d=2$ and $p=4$, and the Cauchy--Schwarz 
		inequality to find
		\begin{equation*}
			\begin{split}
				|T_3| 
				&
				\le C \norm{w_h}_{L^2(\Omega)}\tnorm{\boldsymbol{u}_h}_v
				\tnorm{\boldsymbol{v}_h}_v.
			\end{split}
		\end{equation*}
		Summing the bounds on $T_1$, $T_2$, and $T_3$ yields the result.

		The proof of \cref{eq:o_space_bnd_2} differs in the cases of $d=2$ and 
		$d=3$. We
		begin with $d=3$. To bound $T_1$, we first apply the generalized H\"{o}lder 
		inequality with $p =3$,  
		$q = 6$ and $r=2$ followed by the Cauchy--Schwarz inequality to find:
		\begin{equation*}
			\begin{split}
				|T_1| 
				& \le  \norm{w_h}_{L^3(\Omega)} \norm{u_h}_{L^6(\Omega)}
				\tnorm{\boldsymbol{v}_h}_{v} .
			\end{split} 
		\end{equation*}
		Now applying \cref{eq:L3_bnd_3d} and \cref{eq:disc_sobolev} with 
		$q=6$,
		we have
		\begin{equation*}
			|T_1| \le  C 
			\norm{w_h}_{L^2(\Omega)}^{1/2} \tnorm{\boldsymbol{w}_h}_{v}^{1/2} 
			\tnorm{\boldsymbol{u}_h}_v \tnorm{\boldsymbol{v}_h}_v.
		\end{equation*}
		To bound $T_2$, we apply the generalized H\"{o}lder inequality with $p 
		=3$,  
		$q = 6$ and $r=2$ to find
		\begin{equation*}
			|T_2| \le  \sum_{K \in \mathcal{T}_h}  \norm{w_h\cdot n}_{L^3(\partial 
				K)} 
			\norm{u_h}_{L^6(\partial K)} 
			\norm{v_h - \bar{v}_h}_{L^2(\partial K)} .
		\end{equation*}
		Next, using the local discrete trace inequality $\norm{v_h}_{L^p(\partial 
			K)} \le C h_K^{-1/p}\norm{v_h}_{L^p(K)}$ (see e.g. \cite[Lemma 
		1.52]{Pietro:book}) for $p=3$ and $p=6$, the Cauchy--Schwarz inequality,  
		\cref{eq:L3_bnd_3d}, and \cref{eq:disc_sobolev} with 
		$q=6$, we have
		\begin{equation*}
			\begin{split}
				|T_2| 
				&\le C
				\norm{w_h}_{L^2(\Omega)}^{1/2} \tnorm{\boldsymbol{w}_h}_{v}^{1/2} 
				\tnorm{\boldsymbol{u}_h}_v \tnorm{\boldsymbol{v}_h}_v.
			\end{split}
		\end{equation*}
		To bound $T_3$, we again apply the generalized H\"{o}lder inequality with 
		$p 
		=3$,  
		$q = 6$ and $r=2$:
		\begin{equation*}
			\begin{split}
				|T_3| &\le \sum_{K \in \mathcal{T}_h} 
				\norm{w_h\cdot n}_{L^3(\partial K)} \norm{u_h - 
					\bar{u}_h}_{L^6(\partial K)} 
				\norm{v_h - \bar{v}_h}_{L^2(\partial K)}.
			\end{split}
		\end{equation*}
		Now, applying the local discrete trace inequality $\norm{v_h \cdot 
			n}_{L^p(\partial K)} \le C h_K^{-1/p} \norm{v_h}_{L^p(K)}$ with 
		$p=3$, \cref{lem:facet_scaling} with $d = 3$ and  $p=6$, the discrete 
		Cauchy--Schwarz inequality, and \cref{eq:L3_bnd_3d}, we have
		\begin{equation*}
			\begin{split}
				|T_3| 
				&\le  C \norm{w_h}_{L^2(\Omega)}^{1/2} 
				\tnorm{\boldsymbol{w}_h}_v^{1/2} \tnorm{\boldsymbol{u}_h}_v 
				\tnorm{\boldsymbol{v}_h}_v.
			\end{split}
		\end{equation*}
		Summing the bounds on $T_1$, $T_2$, and $T_3$ yields the result.
		The case for $d=2$ follows similarly, instead using $p=q=4$ and $r=2$ in 
		the generalized H\"{o}lder inequality, \cref{eq:disc_sobolev} with $q=4$, 
		and \cref{lem:facet_scaling} with $d=2$ and $p=4$. 
	\end{proof}

	\subsubsection{Proof of \Cref{lem:linfty_bnd}}
	
	\begin{proof}
		Let $\tilde{\boldsymbol{u}}_h$ be the exponential interpolant of 
		$\boldsymbol{u}_h$ as defined in  \cref{eq:exp_interp_property}. Testing 
		\cref{eq:discrete_problem} with 
		$\tilde{\boldsymbol{u}}_h$,
		integrating by parts in time, and using the defining properties of the 
		exponential 
		interpolant, 
		we 
		have
		\begin{equation} \label{eq:energy_improve_inter1}
			\begin{split}
				\frac{1}{2} \int_{I_n} \od{}{t} 
				\norm{u_h(t)}_{L^2(\Omega)}^2e^{-\lambda(t-t_n)} 
				\dif t + 
				\norm[1]{u_{n}^+}_{L^2(\Omega)}^2& + 
				\int_{I_n} ( \nu a_h(\boldsymbol{u}_h, \tilde{\boldsymbol{u}}_h) + 
				o_h(u_h; 
				\boldsymbol{u}_h, \tilde{\boldsymbol{u}}_h) ) 
				\dif t 
				\\ \notag
				&= (u_{n}^-,u_{n}^+)_{\mathcal{T}_h} + \int_{I_n} (f, 
				\tilde{\boldsymbol{u}}_h)_{\mathcal{T}_h} 
				\dif t.
			\end{split}
		\end{equation}
		Integrating by parts again in time and 
		applying the Cauchy--Schwarz inequality and Young's inequality to the first term on the 
		right 
		hand side,
		we have
		\begin{equation*} 
			\begin{split}
				\frac{\lambda}{2}\int_{I_n}& \norm{u_h(t)}_{L^2(\Omega)}^2 
				e^{-\lambda(t-t_n)} \dif t 
				+ 
				\frac{1}{2}\norm[1]{u_{n+1}^-}^2_{L^2(\Omega)} e^{-\lambda\Delta 
					t}   
				\\
				&\le \frac{1}{2}\norm[1]{u_{n}^-}_{L^2(\Omega)}^2+ \int_{I_n} (f, 
				\tilde{u}_h)_{\mathcal{T}_h} 
				\dif t-\int_{I_n} ( \nu a_h(\boldsymbol{u}_h, 
				\tilde{\boldsymbol{u}}_h) + 
				o_h(u_h; 
				\boldsymbol{u}_h, \tilde{\boldsymbol{u}}_h) ) 
				\dif t .
			\end{split}
		\end{equation*}
		We now focus on bounding the right hand side. By the boundedness of 
		$a_h(\cdot,\cdot)$ \cref{eq:ah_coer_bnd} and \cref{eq:lp_bnd_exp_interp}, there exists a constant 
		$C_1>0$ such that
		\begin{equation*}
			\int_{I_n}  | a_h(\boldsymbol{u}_h,\tilde{\boldsymbol{u}}_h)| \dif t  
			\le C_1 \int_{I_n} \tnorm{\boldsymbol{u}_h}_v^2 \dif t.
		\end{equation*}
		In two spatial dimensions, we can use \Cref{lem:o_space_bnd}, 
		\cref{eq:lp_bnd_exp_interp}, and hence Young's inequality with some $\epsilon_1 > 0$ to
		find there exists a constant $C_2>0$ such that
		\begin{equation*}
			\begin{split}
				\int_{I_n} 
				o_h(u_h; 
				\boldsymbol{u}_h, \tilde{\boldsymbol{u}}_h)  
				\dif t \le 
				\frac{\epsilon_1}{2}\norm{u_h}_{L^{\infty}(I_n;L^2(\Omega))}^2 
				+ 
				\frac{C_2}{2\epsilon_1}\del[2]{\int_{I_n}  
					\tnorm{\boldsymbol{u}_h}_v^2 \dif t}^2.
			\end{split}
		\end{equation*}
		Next by the Cauchy--Schwarz inequality, Young's inequality, the discrete 
		Poincar\'{e} inequality, and \cref{eq:lp_bnd_exp_interp}, there exists a 
		constant $C_3> 0$ such that for some $\epsilon_2 > 0$,
		\begin{equation*}
			\begin{split}
				\int_{I_n} (f, 
				\tilde{u}_h)_{\mathcal{T}_h} \dif t
				\le \frac{1}{2\epsilon_2}\int_{I_n} \norm{f}_{L^2(\Omega)}^2 \dif 
				t + 
				\frac{C_3\epsilon_2}{2}\int_{I_n} 
				\tnorm{\boldsymbol{u}_h}_{v}^2 \dif t.
			\end{split}
		\end{equation*}
		Thus,
		\begin{equation*} 
			\begin{split}
				&\frac{\lambda}{2}\int_{I_n} \norm{u_h}_{L^2(\Omega)}^2 
				e^{-\lambda(t-t_n)} \dif t
				+ 
				\frac{1}{2}\norm[1]{u_{n+1}^-}_{L^2(\Omega)} e^{-\lambda\Delta t} - 
				\frac{\epsilon_1}{2}\norm{u_h}_{L^{\infty}(I_n;L^2(\Omega))}^2
				\\
				&\le  \frac{1}{2}\norm[1]{u_{n}^-}_{L^2(\Omega)}^2  + 
				\frac{1}{2\epsilon_2}\int_{I_n} \norm{f}_{L^2(\Omega)}^2 \dif t + 
				\frac{C_3\epsilon_2}{2}\int_{I_n} 
				\tnorm{\boldsymbol{u}_h}_{v}^2 \dif t   
				+ C_1\nu 
				\int_{I_n} 
				\tnorm{\boldsymbol{u}_h}_v^2 \dif t +
				\frac{C_2}{2\epsilon_1}\del[2]{\int_{I_n}  
					\tnorm{\boldsymbol{u}_h}_v^2 \dif t}^2.
			\end{split}
		\end{equation*}
		Choosing $\lambda = 1 /  \Delta t$ and applying the scaling identity in
		\cref{eq:time_scaling} with $p=\infty$, we find there exists a constant $C_4>0$ 
		such that
		\begin{equation*} 
			\begin{split}
				& \del{\frac{C_4^{-1}e^{-1}}{2} -\frac{\epsilon_1}{2}} 
				\norm{u_h}_{L^{\infty}(I_n;L^2(\Omega))}^2  \\
				&\le \frac{1}{2}\norm[1]{u_{n}^-}_{L^2(\Omega)}^2+ 
				\frac{1}{2\epsilon_2}\int_{I_n} \norm{f}_{L^2(\Omega)}^2 \dif t + 
				\frac{C_3\epsilon_2}{2}\int_{I_n} 
				\tnorm{\boldsymbol{u}_h}_{v}^2 \dif t  + 
				C_1\nu 
				\int_{I_n} 
				\tnorm{\boldsymbol{u}_h}_v^2 \dif t +
				\frac{C_2}{2\epsilon_1}\del[2]{\int_{I_n}  
					\tnorm{\boldsymbol{u}_h}_v^2 \dif t}^2.
			\end{split}
		\end{equation*}
		Choosing $\epsilon_1 =C_4^{-1}e^{-1}/2$, $\epsilon_2 =  2\nu$, using
		the a priori estimates on $\boldsymbol{u}_h$ in \Cref{lem:energy_stability},
		and rearranging, we see there exists
		a $C_5> 0$ such that
		\begin{equation*} 
			\begin{split}
				&
				\norm{u_h}_{L^{\infty}(I_n;L^2(\Omega))}^2
				\\
				&\le C_5\del{ \frac{1}{\nu}\int_{0}^{T}\norm{f}_{L^2(\Omega)}^2 
					\dif 
					t \; + \;
					\norm[1]{u_{0}}_{L^2(\Omega)}^2} +
				\frac{C_5}{\nu^2}\del[2]{ 
					\frac{1}{\nu}\int_{0}^{T}\norm{f}_{L^2(\Omega)}^2 
					\dif t \; + \;
					\norm[1]{u_{0}}_{L^2(\Omega)}^2}^2.
			\end{split}
		\end{equation*}
		This bound holds uniformly for every space-time slab, so
		the result follows.
		
	\end{proof}
	
	\subsubsection{Proof of \Cref{thm:uniqueness_HO}.}
	
	
	\begin{proof}
		Consider an arbitrary space-time slab $\mathcal{E}^m$.
		Suppose $(u_1,\bar{u}_1) \in \mathcal{V}^{\text{div}}_h \times 
		\bar{\mathcal{V}}_h $ 
		and $(u_2,\bar{u}_2) \in \mathcal{V}^{\text{div}}_h \times 
		\bar{\mathcal{V}}_h $ are 
		two solutions to~\cref{eq:discrete_problem} corresponding to the same 
		problem data $f$ and $u_0$, and set $\boldsymbol{w}_h = 
		\boldsymbol{u}_1 - \boldsymbol{u}_2$. 
		Then, for all $\boldsymbol{v}_h \in \mathcal{V}_h^{\text{div}} \times \bar{\mathcal{V}}_h$, it holds that
		\begin{equation}\label{eq:highsol1-sol2_a}
			\begin{split}
				-\int_{I_m} (w_h, \partial_t v_h)_{\mathcal{T}_h} \dif t +  
				(w_{m+1}^-,  v_{m+1}^-)_{\mathcal{T}_h}  +
				\int_{I_m} \nu a_h(\boldsymbol{w}_h,\boldsymbol{v}_h) \dif t
				\\ +
				\int_{I_m} \del{o_h(u_1,\boldsymbol{u}_1,\boldsymbol{v}_h)
					- o_h(u_2,\boldsymbol{u}_2,\boldsymbol{v}_h)} \dif t
				= (w_m^-,  v_m^+)_{\mathcal{T}_h}.
			\end{split}
		\end{equation}
		\textbf{\emph{Step one:}}
		Testing \cref{eq:highsol1-sol2_a} with 
		$\boldsymbol{v}_h = \boldsymbol{w}_h$, integrating by parts in time,  
		using the coercivity of $a_h(\cdot,\cdot)$, and noting that $o_h(u_2,\boldsymbol{u}_2,\boldsymbol{w}_h) -
		o_h(u_1,\boldsymbol{u}_1,\boldsymbol{w}_h)	\le 
		-o_h(w_h,\boldsymbol{u}_2,\boldsymbol{w}_h)$ by \cref{eq:ohequals}, we find
		\begin{equation*}
			\begin{split}
				\frac{1}{2} &\norm[1]{w_{m+1}^-}_{L^2(\Omega)}^2  +	
				\frac{1}{2}\norm[1]{ 
					\sbr{w_h}_m}_{L^2(\Omega)}^2 - 
				\frac{1}{2}\norm[1]{w_{m}^-}_{L^2(\Omega)}^2  + \;
				C \nu \int_{I_m} \tnorm{\boldsymbol{w}_h}_v^2 \dif t \le
				\int_{I_m} |o_h(w_h,\boldsymbol{u}_2,\boldsymbol{w}_h)|  \dif t.
			\end{split}
		\end{equation*}
		Summing over all space-time slabs $n = 0, \dots, N-1$, rearranging, and noting that 
		$w_{0}^- = 0$, we see that there exists a $C_1>0$ such that
		\begin{equation} \label{eq:step_one_uniq}
			\begin{split}
				&\norm[1]{w_{N}^-}_{L^2(\Omega)}^2  +	
				\sum_{n=0}^{N-1}\norm[1]{ 
					\sbr{w_h}_n}_{L^2(\Omega)}^2 
				+ \;
				\nu \int_{0}^{T} \tnorm{\boldsymbol{w}_h}_v^2 \dif t \le
				C_1\int_{0}^{T} |o_h(w_h,\boldsymbol{u}_2,\boldsymbol{w}_h)|
				\dif t
				.
			\end{split}
		\end{equation}
		
		\noindent
		\textbf{\emph{Step two:}} Fix an integer $m$ such that $0 \le m  \le N-1$. Testing \cref{eq:highsol1-sol2_a} with the discrete 
		characteristic function $\boldsymbol{v}_h 
		=\boldsymbol{w}_{\chi}$ where $s = \arg \sup_{t \in I_m} 
		\norm[0]{u_h(t)}_{L^2(\Omega)}$, integrating by parts in time, using Young's inequality, we have after rearranging
		\begin{equation*}
			\begin{split}
				& \frac{1}{2}\sup_{t \in I_{m}}\norm[1]{w_h(t) }_{L^2(\Omega)}^2 - 
				\frac{1}{2}\sup_{t 
					\in I_{m-1}}\norm[1]{w_h(t) }_{L^2(\Omega)}^2 \\ &   
				\le - \int_{I_m} \del{\nu 
					a_h(\boldsymbol{w}_h,\boldsymbol{w}_{\chi})  +
					o_h(u_1,\boldsymbol{u}_1,\boldsymbol{w}_{\chi})
					- o_h(u_2,\boldsymbol{u}_2,\boldsymbol{w}_{\chi})} \dif t,
			\end{split}
		\end{equation*}
		where we have used that $\sup_{t \in I_{m-1}} 
		\norm[0]{w(s)}_{L^2(\Omega)} \ge \norm[0]{w_m^-}_{L^2(\Omega)}$.
		Setting $I_{-1} = 
		\cbr{t_0} = 
		\cbr{0}$, we can sum over the space-time slabs $n = 0,\dots,m$ and 
		use the boundedness of $a_h(\cdot,\cdot)$ and the bound
		\cref{eq:bnd_char} to find there exists a constant $C_2>0$ such that
		\begin{equation} \label{eq:uniq_inter}
			\begin{split}
				& \frac{1}{2}\sup_{t \in I_{m}}\norm[1]{w_h(t) }_{L^2(\Omega)}^2    
				\le C_2 \nu\int_{0}^{T} \tnorm{\boldsymbol{w}_h}_v^2 \dif t  + 
				\int_{0}^{T}  
				\left|o_h(u_2,\boldsymbol{u}_2,\boldsymbol{w}_{\chi})
				-o_h(u_1,\boldsymbol{u}_1,\boldsymbol{w}_{\chi}) \right| \dif t,
			\end{split}
		\end{equation}
		where we have used that $\sup_{t 
			\in I_{-1}}\norm[0]{w_h(t) }_{L^2(\Omega)} = \norm[0]{w_0^- 
		}_{L^2(\Omega)} = 0$. This bound holds uniformly for all space-time slabs,
		and thus we can replace the supremum over $I_n$ in \cref{eq:uniq_inter} with
		the supremum over $[0,T]$.
		Doing so, and adding $2\nu\int_{0}^{T} \tnorm{\boldsymbol{w}_h}_v^2 \dif t$
		to both sides we see there exists a constant $C_3>0$ such that
		\begin{equation*}
			\begin{split}
				\frac{1}{2}	\norm{w_h}^2_{L^{\infty}(0,T;L^2(\Omega))}    
				&+   
				2\nu\int_{0}^{T} \tnorm{\boldsymbol{w}_h}_v^2 \dif t \\
				&\le C_3\int_{0}^{T} 
				\left|o_h(w_h,\boldsymbol{u}_2,\boldsymbol{w}_h)  
				\right| \dif 
				t +
				\int_{0}^{T}  
				\left| o_h(w_h,\boldsymbol{u}_2,\boldsymbol{w}_{\chi}) \right| 
				\dif 
				t + \int_0^T \left|
				o_h(u_1,\boldsymbol{w}_h,\boldsymbol{w}_{\chi}) \right| \dif t.
			\end{split}
		\end{equation*}
		Here, we have used the bound
		\cref{eq:step_one_uniq} from step one, that $o_h(u_2,\boldsymbol{u}_2,\boldsymbol{w}_{\chi}) -
		o_h(u_1,\boldsymbol{u}_1,\boldsymbol{w}_{\chi})	= 
		-o_h(w_h,\boldsymbol{u}_2,\boldsymbol{w}_{\chi}) 
		-
		o_h(u_1,\boldsymbol{w}_h,\boldsymbol{w}_{\chi})$, and the triangle inequality.
		From \Cref{lem:o_space_bnd} and two applications of 
		Young's 
		inequality, first with $p=q=2$ and second with $p = 4, q = 4/3$, we 
		find
		\begin{equation*}
			\begin{split}
				|o_h(w_h,\boldsymbol{u}_2,\boldsymbol{w}_h)|
				& \le C_4\del{ \frac{1}{2\epsilon} 
					\norm{w_h}_{L^2(\Omega)}^2\tnorm{\boldsymbol{u}_2}_{v}^2
					+ 
					\frac{1}{4 \epsilon^3} 
					\norm{w_h}_{L^2(\Omega)}^2\norm{u_2}_{L^2(\Omega)}^2
					\tnorm{\boldsymbol{u}_2}_{v}^2 + 
					\frac{5\epsilon}{4}\tnorm{\boldsymbol{w}_h}^{2}_{v}}.
			\end{split}
		\end{equation*}
		Similarly, from \Cref{lem:o_space_bnd} and \cref{eq:bnd_char}, we have
		\begin{equation*}
			\begin{split}
				|o_h(w_h,\boldsymbol{u}_2,\boldsymbol{w}_{\chi})|
				& \le C_5\del{ \frac{1}{2\epsilon} 
					\norm{w_h}_{L^2(\Omega)}^2\tnorm{\boldsymbol{u}_2}_{v}^2
					+ 
					\frac{1}{4 \epsilon^3} 
					\norm{w_h}_{L^2(\Omega)}^2\norm{u_2}_{L^2(\Omega)}^2
					\tnorm{\boldsymbol{u}_2}_{v}^2 + 
					\frac{5\epsilon}{4}\tnorm{\boldsymbol{w}_h}^{2}_{v}},
			\end{split}
		\end{equation*}
		and finally,
		\begin{equation*}
			\begin{split}
				\left|
				o_h(u_1,\boldsymbol{w}_h,\boldsymbol{w}_{\chi}) \right| & \le C_6 \del{ 
					\norm{u_1}_{L^2(\Omega)} \tnorm{\boldsymbol{w}_h}_v^2 + 
					\frac{1}{4\epsilon^3} \norm{u_1}_{L^2(\Omega)}^{2} 
					\norm{w_h}_{L^2(\Omega)}^{2} 
					\tnorm{\boldsymbol{u}_1}_v^{2} + 
					\frac{3\epsilon}{4}\tnorm{\boldsymbol{w}_h}_v^{2}
				},
			\end{split}
		\end{equation*}
		where $\epsilon >0$.
		Collecting the above bounds, choosing $\epsilon = O(\nu)$ sufficiently small and rearranging, we can find a $C_{7} > 0$ 
		such 
		that 
		\begin{equation} \label{eq:uniq_inter2}
			\begin{split}
				&\norm{w_h}^2_{L^{\infty}(0,T;L^2(\Omega))}    
				+   
				\nu\int_{0}^{T} \tnorm{\boldsymbol{w}_h}_v^2 \dif t \\
				&\le 
				\frac{C_{7}}{\nu^4} \del{ \nu^3
					\norm{u_1}_{L^{\infty}(0,T;L^2(\Omega))}  +
					\nu\norm{u_1}_{L^{\infty}(0,T;L^2(\Omega))}^2 
					\int_{0}^T\tnorm{\boldsymbol{u}_1}_{v}^2 \dif t +
					\del{\nu^3
						+ 
						\nu\norm{u_2}_{L^{\infty}(0,T;L^2(\Omega))}^2}\int_{0}^T\tnorm{\boldsymbol{u}_2}_{v}^2
					\dif t} \\& \times\del{\norm{w_h}_{L^{\infty}(0,T;L^2(\Omega))}^2 +
					\nu\int_{0}^{T}\tnorm{\boldsymbol{w}_h}^{2}_{v} \dif t 
				}.
			\end{split}
		\end{equation}
		\\
		
		\noindent
		\textbf{\emph{Step three:}} For notational convenience, let $\Xi = \nu^{-1} \int_{0}^T\norm{f}_{L^2(\Omega)}^2 \dif t + \norm[0]{u_0}_{L^2(\Omega)}^2$.
		Applying the bounds in \Cref{lem:energy_stability} and \Cref{lem:linfty_bnd} to \cref{eq:uniq_inter2}, we find there exists a 
		$C_{8} > 0$ such that 
		\begin{equation*}
			\begin{split}
				&\norm{w_h}^2_{L^{\infty}(0,T;L^2(\Omega))}    
				+   
				\nu\int_{0}^{T} \tnorm{\boldsymbol{w}_h}_v^2 \dif t \\
				&\le 
				\frac{C_{8}}{\nu^4} \del{\nu^{3} \Xi^{1/2} + 
					\nu^2 \Xi   
					+ \Xi^2 + 
					\nu^{-2} \Xi^3   } 
				\del{\norm{w_h}_{L^{\infty}(0,T;L^2(\Omega))}^2 +
					\nu\int_{0}^{T}\tnorm{\boldsymbol{w}_h}^{2}_{v} \dif t 
				}.
			\end{split}
		\end{equation*}
		The result follows if $\nu < 1$ and $\Xi \le \frac{1}{4}\min 
		\cbr[0]{C_{8}^{-1/3},C_{8}^{-1/2}, C_{8}^{-1},C_{8}^{-2}} 
		\nu^2$.
	\end{proof}

	\subsection{Recovering the pressure}
	
	Existence of the pressure pair $(p_h, \bar{p}_h) \in 
	\boldsymbol{\mathcal{Q}}_h$ satisfying
	\cref{eq:discrete_problem} will require the 
	following inf-sup 
	condition:
	\begin{thm}[Inf-sup condition] \label{thm:inf_sup}
		Suppose that the spatial mesh $\mathcal{T}_h$ is conforming and quasi-uniform.
		There exists a constant $\beta > 0$, independent of $h$ and $\Delta t$,
		such that for all $\boldsymbol{q}_h \in \boldsymbol{\mathcal{Q}}_h$,
		\begin{equation} \label{eq:inf_sup}
			\sup_{0 \ne \boldsymbol{v}_h \in \boldsymbol{\mathcal{V}}_h}  \frac{\int_{I_n} 
				b_h(\boldsymbol{q}_h,v_h) \dif t}{\del{\int_{I_n} 
					\tnorm{\boldsymbol{v}_h}_{v}^2 
					\dif t}^{1/2}} \ge \beta \del{\int_{I_n} 
				\tnorm{\boldsymbol{q}_h}_p^2 \dif t}^{1/2}.
		\end{equation}
	\end{thm}
	The proof, which exploits the tensor-product structure of the finite element 
	spaces in an essential way, is an extension of the proof of \cite[Lemma 
	1]{Rhebergen:2018b} to 
	the space-time setting. 
	\subsubsection{Proof of \Cref{thm:inf_sup}.}
	
	By \cite[Theorem 3.1]{Howell:2011}, the inf-sup condition 
	\cref{eq:inf_sup}
	is satisfied if we can decompose $b_h(\cdot,\cdot)$ into $b_1(\cdot,\cdot): 
	V_h \times 
	Q_h \to \mathbb{R}$ and $b_2(\cdot,\cdot): V_h \times 
	\bar{Q}_h \to 
	\mathbb{R}$ such that, for some constants
	$\alpha_1,\alpha_2 
	> 0$, it holds that
	\begin{subequations}\label{eq:_inf_sup_b1_b2}
		\begin{align}  \label{eq:_inf_sup_b1}
			\sup_{(v_h,\bar{v}_h) \in \mathcal{Z}_{b_2} \times \bar{\mathcal{V}}_h} \frac{\int_{I_n} 
				b_1(q_h,v_h)\dif 
				t}{\del{\int_{I_n}\tnorm{\boldsymbol{v}_h}_v^2 \dif t}^{1/2}} &\ge 
			\alpha_1 \del[2]{\int_{I_n} 
				\norm{q_h}_{L^2(\Omega)}^2 \dif t}^{1/2}, \\
			&\text{and} \notag \\  \label{eq:_inf_sup_b2}
			\sup_{\boldsymbol{v}_h \in \boldsymbol{\mathcal{V}}_h} \frac{\int_{I_n} 
				b_2(\bar{q}_h,v_h)\dif 
				t}{\del{\int_{I_n}\tnorm{\boldsymbol{v}_h}_v^2 \dif t}^{1/2}}& \ge 
			\alpha_2 
			\del[2]{\sum_{K\in\mathcal{T}_h}\int_{I_n} h_K\norm{\bar{q}}_{\partial K}^2 
				\dif t}^{1/2},
		\end{align}
	\end{subequations}
	where
	\begin{equation*}
		\mathcal{Z}_{b_2} = \cbr[2]{v_h \in \mathcal{V}_h 
			\; : \; 
			\int_{I_n} 
			b_2(v_h,\bar{q}_h)\dif t 
			= 0, 
			\quad \forall \bar{q}_h \in \bar{\mathcal{Q}}_h}.
	\end{equation*}
	We thus define
	\begin{equation*}
		b_1(q_h,v_h) = - \sum_{K\in\mathcal{T}_h}\int_{K}q_h \nabla \cdot v_h \dif x
		\quad \text{and} \quad b_2(\bar{q}_h,v_h) =  
		\sum_{K\in\mathcal{T}_h}\int_{\partial 
			K}v_h \cdot n \bar{q}_h \dif s.
	\end{equation*}
	We begin by proving \cref{eq:_inf_sup_b1}. The tensor-product structure
	of the space $\mathcal{Q}_h$ ensures that we can expand any $q_h \in \mathcal{Q}_h$ in terms
	of an orthonormal basis of $P_k(I_n)$ with respect
	to the $L^2(I_n)$ inner product:
	\begin{equation} \label{eq:qh_expansion}
		q_h = \sum_{i = 0}^k \phi_i(t) q_i(x), \quad \phi_i \in P_{k}(I_n), \; q_i \in Q_h.
	\end{equation}
	Since each $q_i \in 
	L_0^2(\Omega)$, there exist $z_i 
	\in H_0^1(\Omega)^d$, $0 \le i \le k$ and
	constants $\beta_i$, $0 \le i \le k$  such that
	$\nabla \cdot z_i = -q_i$ and $\beta_i \norm{z_i}_{H^1(\Omega)} \le  
	\norm{q_i}_{L^2(\Omega)}$ (see e.g. \cite[Theorem 6.5]{Pietro:book}). We construct the desired $\boldsymbol{\psi}_h = (\psi_h,\bar{\psi}_h) \in  \mathcal{Z}_{b_2} \times \bar{\mathcal{V}}_h$ 
	by choosing
	\begin{equation*}
		\psi_h = \sum_{i=0}^k \phi_i(t) \Pi_{\text{BDM}} z_i, \quad \text{ and } 
		\quad \bar{\psi}_h = \sum_{i=0}^k \phi_i(t) \bar{\Pi}_V z_i,
	\end{equation*}
	where $\Pi_{\text{BDM}}:\sbr{H^1(\Omega)}^d \to V_h$ is the BDM projection \cite{Boffi:book} and $\bar{\Pi}_V$ is the $L^2$ 
	projection onto the
	space $\bar{V}_h$.
	By the orthonormality of the basis $\cbr{\phi_i}_{i=0}^k$, definition
	of the BDM projection \cite{Boffi:book}, the single-valuedness of
	$z_i \cdot n$ and $\bar{q}_i$ across element faces, and the fact that
	$z_i \in H_0^1(\Omega)$, we have
	\begin{equation*}
		b_2(\bar{q}_h,\psi_h) = \sum_{K\in \mathcal{T}_h} \sum_{i=0}^k 
		\int_{\partial 
			K} z_i  \cdot n 
		\bar{q}_i \dif s = 0,
	\end{equation*}
	and thus $\psi_h \in \mathcal{Z}_{b_2}$.
	We now show that $\boldsymbol{\psi}_h$ satisfies the inequality in \cref{eq:_inf_sup_b1} with
	some $\alpha_1>0$ independent of the mesh parameters $h$ and $\Delta t$.
	Given $q_h \in \mathcal{Q}_h$, we can use the expansion \cref{eq:qh_expansion}, the definition
	of $z_i$, $0 \le i \le k$, and the commuting diagram 
	property of the BDM projection \cite{Boffi:book} to find
	\begin{equation}\label{eq:inf_sup_inter_2}
		\begin{split}
			\int_{I_n} \norm{q_h}_{L^2(\Omega)}^2\dif t 
			&=  \int_{I_n} b_1(q_h,\psi_h)\dif t.
		\end{split}
	\end{equation}
	Next, we need to show existence of a constant $\alpha_1> 0$, independent of the
	mesh parameters $h$ and
	$\Delta t$, such that
	\begin{equation} \label{eq:inf_sup_norm_bnd}
		\alpha_1^2 \int_{I_n} \tnorm{\boldsymbol{\psi}_h}_{v}^2 \dif t \le \int_{I_n} 
		\norm{q_h}_{L^2(\Omega)}^2\dif t.
	\end{equation}
	But, this can easily be reduced to the proof of \cite[Lemma 4.5]{Rhebergen:2017} by
	expanding $\psi_h$ in terms of an orthonormal basis of
	$P_k(I_n)$ with respect to the $L^2(I_n)$ inner-product.
	Combining  \cref{eq:inf_sup_inter_2} and \cref{eq:inf_sup_norm_bnd}, we have
	\begin{equation*}
		\frac{\int_{I_n} b_1(q_h,\psi_h) \dif t}{\del{\int_{I_n} 
				\tnorm{\boldsymbol{\psi}_h}_{v}^2  \dif t}^{1/2}} = 
		\frac{\int_{I_n} \norm{q_h}_{L^2(\Omega)}^2 \dif 
			t}{\del{\int_{I_n} \tnorm{\boldsymbol{\psi}_h}_{v}^2\dif t}^{1/2}} \ge 
		\alpha_1^2 \del[2]{\int_{I_n} 
			\norm{q_h}_{L^2(\Omega)}^2 \dif t}^{1/2},
	\end{equation*}
	where $\alpha_1 > 0$ depends on the constants $\beta_i$, $0\le i \le k$.
	
	What we have left to show is \cref{eq:_inf_sup_b2}.
	It suffices to construct an $\boldsymbol{\omega}_h \in 
	\boldsymbol{\mathcal{V}}_h$ such that for some $\alpha_2 > 0$ it holds that
	\begin{equation} \label{eq:infsup_inter2}
		\frac{\int_{I_n} b_2(\bar{q}_h,\omega_h)\dif 
			t}{\del{\int_{I_n}\tnorm{\boldsymbol{\omega}_h}_v^2 \dif t}^{1/2}} \ge 
		\alpha_2
		\del[2]{\sum_{K\in\mathcal{T}_h}\int_{I_n} h_K\norm{\bar{q}_h}_{\partial 
				K}^2 
			\dif t}^{1/2}, \quad \forall \bar{q}_h \in \bar{\mathcal{Q}}_h.
	\end{equation}
	The tensor-product structure
	of $\bar{\mathcal{Q}}_h$ ensures that we can expand any $\bar{q}_h \in \bar{\mathcal{Q}}_h$ in terms
	of an orthonormal basis of $P_k(I_n)$ with respect
	to the $L^2(I_n)$ inner-product:
	\begin{equation} \label{eq:qbar_expansion}
		\bar{q}_h = \sum_{i = 0}^k \phi_i(t) \bar{q}_i(x), \quad \phi_i \in P_{k}(I_n), \; \bar{q}_i \in \bar{Q}_h.
	\end{equation}
	Given $\bar{q}_h \in \bar{\mathcal{Q}}_h$, we construct the required $\boldsymbol{\omega}_h$ by choosing $\bar{\omega}_h = 0$, 
	and defining $\omega_h \in \mathcal{V}_h$ element-wise by:
	\begin{equation*}
		\omega_h|_{K\times I_n} = \sum_{i=0}^k \phi_i(t) L^{\text{BDM}} \del{ 
			\bar{q}_i|_{\partial K}},
	\end{equation*}
	with $\bar{q}_i \in \bar{Q}_h$ defined as in \cref{eq:qbar_expansion}. Here, $L^{\text{BDM}}: P_k(\partial K) \to 
	\del{P_k(K)}^d$ is the local BDM lifting satisfying for all  $\bar{q}_h \in 
	P_k(\partial K)$ (see e.g. \cite[Proposition 
	2.10]{Sayas:book}):
	\begin{equation} \label{eq:bdm_lifting_properties}
		(L^{\text{BDM}} \bar{q}_h) \cdot n = \bar{q}_h, \quad \text{ and } \quad 
		\norm[1]{L^{\text{BDM} 
			}\bar{q}_h}_{L^2(K)} \le Ch_K^{1/2} \norm{\bar{q}}_{L^2(\partial K)}, 
		\quad \forall \bar{q}_h \in 
		P_k(\partial K),
	\end{equation}
	where $n$ is the unit outward normal to $\partial K$.
	Using the first property in \cref{eq:bdm_lifting_properties}, it can be shown that 
	\begin{equation} \label{eq:inf_sup_3_inter}
		\int_{I_n} b_2(\bar{q}_h,\omega_h)\dif t = \sum_{K \in \mathcal{T}_h} 
		\int_{I_n} 
		\norm{\bar{q}_h}_{L^2(\partial K)}^2 \dif t.
	\end{equation}
	The remainder of the proof of \cref{eq:_inf_sup_b2} can easily be reduced to the proof of \cite[Lemma 3]{Rhebergen:2018b} 
	by expanding $\omega_h$ in terms of an orthonormal basis of
	$P_k(I_n)$ with respect to the $L^2(I_n)$ inner-product. In particular,
	it can be shown that $\alpha_2 = C h_{\min}/h_{\max}$, which remains uniformly bounded below
	provided we assume quasi-uniformity of $\mathcal{T}_h$.
	\qed
	\\
	
	\noindent
	The Ladyzhenskaya-Babu\v{s}ka-Brezzi theorem \cite{Boffi:book} yields the following 
	corollary:
	\begin{cor}
		To each discrete velocity solution pair $(u_h,\bar{u}_h) \in \boldsymbol{\mathcal{V}}_h$ 
		guaranteed by \Cref{thm:existence}, 
		there exists a unique discrete pressure pair $(p_h,\bar{p}_h) \in 
		\boldsymbol{\mathcal{Q}}_h$ 
		satisfying 
		\cref{eq:discrete_problem}.
	\end{cor}
	\section{Error analysis for the velocity}
	\label{sec:vel_err}
	
	\subsection{Space-time projection operators}
	Let $ P_h : L^2(\Omega)^d\to V_h^{\text{div}}$ and
	$ \bar{P}_h : L^2(\Gamma)^d \to \bar{V}_h$ denote the orthogonal
	$L^2$-projections onto, respectively, the spaces $V_h^{\text{div}}$
	and $\bar{V}_h$. The approximation properties of $\bar{P}_h$ are
	well-known while the approximation properties of $P_h$ rely critically
	on the fact that $V_h^{\text{div}} \subset H$. In particular, we can
	exploit the best approximation property of the orthogonal projection
	along with the approximation properties of the BDM projection to
	prove (see e.g.  \cite{Rhebergen:2018c}):
	\begin{lem} \label{lem:approx_prop_spat_proj}
		Let $k\ge 1$, $0 \le m \le 2$, and $u \in H^{k+1}(\Omega)^d$. If the spatial mesh $\mathcal{T}_h$ 
		is quasi-uniform and
		consists 
		triangles in two dimensions or tetrahedras in three dimensions, then the following
		estimates hold:
		\begin{align}\label{eq:spat_proj_el}
			\sum_{K \in \mathcal{T}_h} \norm{u - P_h u}^2_{H^m(K)} &\lesssim 
			h^{2(k-m+ 1)}
			|u|_{H^{k+1}(\Omega)}^2,
			\\ \label{eq:spat_proj_face}
			\sum_{K \in \mathcal{T}_h} h_K^{-1}\norm{u - P_h u}^2_{L^2(\partial K)} &\lesssim h^{2k}
			|u|_{H^{k+1}(\Omega)}^2.
		\end{align}
	\end{lem}
	\begin{proof}
		
		We begin by proving \cref{eq:spat_proj_el}. For $m=0$, we have by the best 
		approximation property of the orthogonal $L^2$-projection onto 
		$V_h^{\text{div}}$:
		\begin{equation*} 
			\norm{u - P_h u }_{L^2(\Omega)} = \min_{v_h \in V_h^{\text{div}} } \norm{ u 
				- v_h}_{L^2(\Omega)}.
		\end{equation*}
		Since $\Pi_{\text{BDM}} u \in V_h^{\text{div}}$,
		\cref{eq:spat_proj_el} follows from standard approximation properties
		of the BDM projection \cite{Boffi:book}. The proof for $m = 1$ follows
		by noting that, by triangle inequality,
		\begin{equation*}
			\norm{u - P_h u}_{H^m(K)} \le \norm{u - \Pi_V u}_{H^m(K)} + \norm{\Pi_V u - P_h u}_{H^m(K)},
		\end{equation*}
		where $\Pi_V$ is the orthogonal $L^2$-projection onto $V_h$.
		Using the local inverse inequality $\norm{u_h}_{H^1(K)} \le 
		Ch_K^{-1} \norm{u_h}_{L^2(K)}$, the quasi-uniformity of the spatial mesh 
		$\mathcal{T}_h$, \cref{eq:spat_proj_el}, and the approximation properties of $\Pi_V$
		(see e.g. \cite[Lemma 
		1.58]{Pietro:book}), the result follows. The bound for $m=2$
		follows similarly. To prove \cref{eq:spat_proj_face}, we note that by the local trace inequality 
		for functions in $H^1(K)$, we 
		have
		\begin{equation*}
			\norm{u - P_h u}_{L^2(\partial K)} \le C \del{h_K^{-1/2}\norm{u - P_h u}_{L^2(K)} + \norm{u - P_h 
					u}_{L^2(K)}^{1/2} |u - P_h u|_{H^1(K)}^{1/2}}.
		\end{equation*}
		The result now follows from the quasi-uniformity of the mesh, the 
		Cauchy--Schwarz inequality, and \cref{eq:spat_proj_el}.
	\end{proof}

	Following \cite[Definition 
	4.2]{Chrysafinos:2010}, we introduce a space-time projection operator much in 
	the same spirit as 
	the temporal ``DG-projection" defined in \cite[Eq. (12.9)]{Thomee:book} or 
	\cite[Section 6.1.4]{Dolejsi:book}, but appropriately modified for divergence 
	free fields. Additionally, we will need an analogue of this temporal DG-projection
	onto the facet space $\bar{\mathcal{V}}_h$:
	\begin{defn} \label{def:space_time_kernel_proj}
		\begin{enumerate}
			\item $\mathcal{P}_h: C(I_n;L^2(\Omega)) \to 
			\mathcal{V}_h$ 
			satisfying $(\mathcal{P}_h u)(t_{n+1}^-) = (P_h u)(t_{n+1}^-)$, with 
			$(\mathcal{P}_h u)(t_{0}^-) = P_{h} u(t_0)$, and
			\begin{equation} \label{eq:space_time_kernel_proj}
				\int_{I_n} (u - \mathcal{P}_h u, v_h)_{\mathcal{T}_h}  \dif t = 0 \quad 
				\forall v_h \in P_{k-1}(I_n,V_h^{\text{div}}).
			\end{equation}
			\item $ \bar{ \mathcal{P}}_h : C(I_n; L^2(\Gamma)) \to  
			P_{k}(I_n;\bar{V}_h)$ satisfying $(\bar{\mathcal{P}}_h u 
			)(t_{n+1}^-) = (\bar{P}_h u)(t_{n+1}^-)$
			\begin{equation}  \label{eq:facet_space_time_proj}
				\int_{I_n} ( u - \bar{ \mathcal{P}}_h u, \bar{v}_h  
				)_{\partial \mathcal{T}_h} \dif t = 0, \quad \forall \bar{v}_h 
				\in P_{k-1}(I_n;\bar{V}_h).	
			\end{equation}
		\end{enumerate}
	\end{defn}
	\noindent We summarize the approximation properties of $\mathcal{P}_h$
	and $\bar{\mathcal{P}}_h$ in  
	\Cref{appendix:projection_estimates}.
	\subsection{Parabolic Stokes projection}
	Motivated by~\cite[Definition 4.2]{Chrysafinos:2010}, we introduce a parabolic 
	Stokes projection 
	which will be crucial to our error analysis in \Cref{sec:vel_err}:
	\begin{defn}[Parabolic Stokes projection] \label{def:stokes_proj}
		Let $u$ be the strong velocity solution to the Navier--Stokes system
		\cref{eq:ns_equations} guaranteed by \Cref{thm:strong_solution}.
		We define the parabolic Stokes projection $(\Pi_h u, \bar{\Pi}_h u, \Pi_h 
		p, \bar{\Pi}_h p) \in \mathcal{V}_h^{\text{div}} \times \bar{\mathcal{V}}_h \times \boldsymbol{\mathcal{Q}}_h$
		to be the solution to the following 
		space-time HDG scheme:
		\begin{equation} \label{eq:stokes_proj}
			\begin{split}
				-&\int_{I_n} (\Pi_h u, \partial_t v_h)_{\mathcal{T}_h} \dif t + 
				((\Pi_h u)_{n+1}^-,v_{n+1}^-)_{\mathcal{T}_h} +
				\int_{I_n} ( \nu a_h(\boldsymbol{\Pi}_h u, \boldsymbol{v}_h)  + 
				b_h(\boldsymbol{\Pi}_h p, v_h) ) \dif t \\
				&= ((\Pi_h u)_n^-,v_{n}^+)_{\mathcal{T}_h} +\int_{I_n} (\partial_t 
				u,  
				v_h)_{\mathcal{T}_h} \dif t +
				\int_{I_n} \nu a_h(\boldsymbol{u}, \boldsymbol{v}_h)\dif t  \quad 
				\forall 
				\boldsymbol{v}_h \in \boldsymbol{\mathcal{V}}_h, \\
				&\int_{I_n} b_h(\boldsymbol{q}_h, \Pi_h u) \dif t = 0 \quad \forall 
				\boldsymbol{q}_h \in 
				\boldsymbol{\mathcal{Q}}_h,
			\end{split}
		\end{equation}
		where $(\Pi_h u)_0^- = P_h u(t_0)$ and $(\bar{\Pi}_h u)_0^-$ may be 
		arbitrarily chosen. Here, we have denoted $\boldsymbol{\Pi}_h u = (\Pi_h u, 
		\bar{\Pi}_h 
		u)$ and $\boldsymbol{\Pi}_h p = (\Pi_h p, \bar{\Pi}_h p)$.
	\end{defn}
	
	\begin{rem}
		We remark that \cref{eq:stokes_proj} is simply 
		a space-time HDG scheme for the
		evolutionary Stokes problem \cref{eq:stokes_equations} with $f = u_t - \nu\Delta u$ and
		$u_0 = u(0)$. Consequently, $\Pi_h u \in 
		\mathcal{V}_h^{\text{div}}$ and thus $\Pi_h u \in H$.
	\end{rem}
	
	\subsection{Uniform bounds on the parabolic Stokes projection}
	To perform our error analysis in \Cref{sec:vel_err}, we will require
	a uniform bound on the Stokes projection:
	\begin{equation*}
		\text{ess\;sup}_{0 
			\le t 
			\le T} \tnorm{\boldsymbol{\Pi}_h u}_{\nu} \le C(u,u_0,\nu).
	\end{equation*}
	Our plan is to follow the proof of \cite[Theorem
	4.10]{Chrysafinos:2010}.  Therein, an essential ingredient is a
	discrete Stokes operator.  Unfortunately, as
	$(\cdot,\cdot)_{\mathcal{T}_h}$ is \emph{not} an inner-product on
	$\boldsymbol{V}_h$, we cannot leverage the Riesz representation
	theorem to infer the existence of a discrete Stokes operator in the
	HDG setting.  Instead, we introduce a novel discrete Stokes operator
	by mimicking the static condensation that occurs for the HDG method at
	the algebraic level following ideas from \cite{Calo:2019}.

	\subsubsection{Discrete Stokes operator}
	Consider the variational
	problem: find $\boldsymbol{\phi}_h \in V_h^{\text{div}} \times \bar{V}_h$ such 
	that
	\begin{equation*}\label{eq:disc-lapl-like_inv}
		a_h(\boldsymbol{\phi}_h,\boldsymbol{w}_h) = (u_h, w_h)_{\mathcal{T}_h}, 
		\quad 
		\forall \boldsymbol{w}_h \in V_h^{\text{div}} \times \bar{V}_h.
	\end{equation*} 
	This problem is well-posed by the Lax--Milgram theorem, implying 
	the existence of a well-defined solution operator $S_h: V_h 
	\to V_h^{\text{div}} \times \bar{V}_h$
	such that $\boldsymbol{\phi}_h = S_h (u_h)$.
	Note that $S_h$ \emph{need 
		not be surjective} onto the product space $V_h^{\text{div}} \times \bar{V}_h$. However, as in 
	\cite{Calo:2019},
	we can split the solution operator $S_h$ into ``element" and ``facet" solution 
	operators $S_{\mathcal{K}}$ and $S_{\mathcal{F}}$. We will 
	show that $S_{\mathcal{K}}$
	is invertible.
	
	Define the \emph{facet solution operator} $S_{\mathcal{F}}: V_h^{\text{div}} \to 
	\bar{V}_h$ as the 
	unique solution
	of
	\begin{equation}\label{eq:def_facet_op}
		a_h((v_h, S_{\mathcal{F}}(v_h)),(0,\bar{w}_h)) = 0, \quad \forall \bar{w}_h 
		\in 
		\bar{V}_h.
	\end{equation}
	Since $a_h(\cdot,\cdot)$ is symmetric, $S_{\mathcal{F}}$ is self-adjoint. Next, we introduce a new bilinear form on $V_h^{\text{div}} \times 
	V_h^{\text{div}}$:
	\begin{equation} \label{eq:astar_def}
		a_h^{\star}(v_h, w_h) = a_h((v_h, S_{\mathcal{F}}(v_h)),(w_h, 
		S_{\mathcal{F}}(w_h))),
	\end{equation}
	for which we introduce the \emph{element solution operator} $S_{\mathcal{K}} : 
	V_h^{\text{div}}\to 
	V_h^{\text{div}}$ satisfying
	\begin{equation*}
		a_h^{\star}(S_{\mathcal{K}}(u_h),w_h) = (u_h,w_h)_{\mathcal{T}_h}.
	\end{equation*}
	It can be shown that $S_h(u_h) = (S_{\mathcal{K}}(u_h),(S_{\mathcal{F}} \circ 
	S_{\mathcal{K}})(u_h))$ (see \cite[Lemma 3.1]{Calo:2019}).
	We observe that $S_{\mathcal{K}}: V_h^{\text{div}} \to V_h^{\text{div}}$ is 
	injective. By the Rank-Nullity theorem, $S_{\mathcal{K}}$ is bijective.
	Therefore, we can define an inverse operator $A_h = S_{\mathcal{K}}^{-1}$
	satisfying
	\begin{equation*} \label{eq:def_A_h}
		a_h^{\star}(u_h,w_h) = (A_h u_h, w_h)_{\mathcal{T}_h}, \quad \forall w_h 
		\in 
		V_h^{\text{div}}.
	\end{equation*}
	By \cref{eq:astar_def}, we have equivalently that
	\begin{equation} \label{eq:A_h_main_prop}
		a_h((u_h, S_{\mathcal{F}}(u_h)),(w_h, S_{\mathcal{F}}(w_h))) = (A_h u_h, 
		w_h)_{\mathcal{T}_h}, \quad \forall w_h \in V_h^{\text{div}}.
	\end{equation}

	\begin{lem} \label{lem:oper_facet_equiv}
		Fix
		a space-time slab $\mathcal{E}^n$.
		Let $(\Pi_h u,\bar{\Pi}_h u)$ be the velocity components 
		of the parabolic 
		Stokes projection solving \cref{eq:stokes_proj}, and let
		$S_{\mathcal{F}}: V_h^{\text{div}} \to 
		\bar{V}_h$ be the facet solution operator
		introduced in \cref{eq:def_facet_op}. Then, it holds that
		\begin{equation}
			\bar{\Pi}_h u=S_{\mathcal{F}}(\Pi_h u).
		\end{equation}
	\end{lem}
	
	\begin{proof}
		Set 
		$(v_h,\bar{v}_h,q_h,\bar{q}_h) = (0,\bar{v}_h,0,0)$ in 
		\cref{eq:stokes_proj} and expand $\Pi _h u$, $\bar{\Pi}_h u$, and $\bar{v}_h$ in terms of an
		orthonormal 
		basis $\cbr[0]{\phi_i(t)}_{i=0}^k$ of $P_k(I_n)$ with respect to the $L^2(I_n)$ inner-product
		to find 
		\begin{equation*}
			\begin{split}
				\sum_{i=0}^k \sum_{K \in \mathcal{T}_h} \int_{\partial K} 
				\frac{\alpha}{h_K} (\bar \Pi_h u)_i \cdot \bar{v}_i \dif s & = 
				\sum_{i=0}^k 
				\sum_{K \in \mathcal{T}_h} \int_{\partial K} 
				\del{\frac{\alpha}{h_K} (\Pi_h u)_i - 
					\frac{\partial (\Pi_h u)_i}{\partial n}} \cdot \bar{v}_i \dif s.
			\end{split}
		\end{equation*}
		By the definition of the operator $S_{\mathcal{F}}$, we have
		for each $i=0,\dots,k$,
		\begin{equation*}
			\sum_{K \in \mathcal{T}_h} \int_{\partial K} \del{\frac{\alpha}{h_K} 
				(\Pi_h u)_i - 
				\frac{\partial (\Pi_h u)_i}{\partial n}} \cdot \bar{v}_i \dif s = 
			\sum_{K 
				\in 
				\mathcal{T}_h} \int_{\partial K} \frac{\alpha}{h_K} S_{\mathcal{F}} 
			((\Pi_h u)_i)\cdot \bar{v}_i \dif s, 
		\end{equation*}
		and moreover, each $S_{\mathcal{F}}((\Pi_h u)_i)$ is unique. 
		Choosing $\bar{v}_i  = (\bar \Pi_h u)_i- S_{\mathcal{F}}((\Pi_h u)_i) \in 
		\bar{V}_h$ and rearranging
		allows us to conclude $(\bar \Pi_h u)_i = S_{\mathcal{F}}((\Pi_h u)_i)$ for 
		each 
		$i=0,\dots,k$. The result follows by uniqueness of the expansions of $u_h$ and $\bar{u}_h$
		with respect to the chosen basis of $P_k(I_n)$ and the linearity of $S_{\mathcal{F}}$.
	\end{proof}
	\begin{lem}\label{lem:disc_lap_properties}
		Fix
		a space-time slab $\mathcal{E}^n$. Let $(\Pi_h u,\bar{\Pi}_h u) \in \mathcal{V}_h^{\text{div}} \times \bar{\mathcal{V}}_h$ be the velocity components 
		of the parabolic 
		Stokes projection solving \cref{eq:stokes_proj} and
		let $A_h: V_h^{\text{div}} \to V_h^{\text{div}}$ be the discrete Stokes operator satisfying \cref{eq:A_h_main_prop}.
		For notational convenience, we denote $A_h \boldsymbol{\Pi}_h u = (A_h \Pi_h u, A_h\bar \Pi_h u )$. 
		Then, for all $t \in I_n$, it holds that:
		\begin{align} 
			a_h(\boldsymbol{\Pi}_h u,A_h \boldsymbol{\Pi}_h u) &= 
			\norm{A_h \Pi_h u}_{L^2(\Omega)}^2, \label{eq:A_h_property1} \\
			(\partial_t \Pi_h u, A_h \Pi_h u)_{\mathcal{T}_h}&= 
			\frac{1}{2} 
			\od{}{t} a_h(\boldsymbol{\Pi}_h u,\boldsymbol{\Pi}_h u). \label{eq:A_h_property2}
		\end{align}
	\end{lem}
	\begin{proof}
		By \Cref{lem:oper_facet_equiv} and the linearity of $A_h$ and $S_{\mathcal{F}}$, we find
		\begin{equation*}
			S_{\mathcal{F}} \del{ A_h \Pi_h u} = A_h \del{S_{\mathcal{F}} \Pi_h u} = A_h \bar{\Pi}_h u, \quad
			S_{\mathcal{F}} \del{ \partial_t  \Pi_h u} =   \partial_t  \del{ S_{\mathcal{F}} \Pi_h u} = \partial_t \bar{\Pi}_h u.
		\end{equation*}
		The conclusion follows from \cref{eq:A_h_main_prop} after some basic calculations.
	\end{proof}
	
	\subsubsection{Bounding the Stokes projection} 
	
	\begin{lem}[Uniform bound on the Stokes projection] \label{lem:stokes_proj_bnd}
		Let $u$ be the strong velocity solution to the Navier--Stokes system
		\cref{eq:ns_equations} guaranteed by \Cref{thm:strong_solution} and let $(\Pi_h u, \bar{\Pi}_h u, \Pi_h 
		p, \bar{\Pi}_h p)\in \mathcal{V}_h^{\text{div}} \times \bar{\mathcal{V}}_h \times \boldsymbol{\mathcal{Q}}_h$
		be the solution to \cref{eq:stokes_proj}, 
		where we set $\bar{u}_0^- = \bar{P}_h u_0 $. Then, it holds that
		\begin{equation*}
			\norm{\boldsymbol{\Pi}_h 
				u}_{L^{\infty}(0,T;\boldsymbol{\mathcal{V}}_h)}^2  
			\le C \del{ 
				\frac{1}{\nu} \int_{0}^T \norm{\partial_t u - \nu \Delta  
					u}_{L^2(\Omega)}^2 \dif t  
				+ 
				\norm{u_0}_{H^1(\Omega)}^2}.
		\end{equation*}
		Here, we define $\norm{\boldsymbol{\Pi}_h 
			u}_{L^{\infty}(0,T;\boldsymbol{\mathcal{V}}_h)} \equiv \text{ess\;sup}_{0 
			\le t 
			\le T} \tnorm{\boldsymbol{\Pi}_h u}_{\nu}$.
	\end{lem}
	
	\begin{proof}
		The proof will proceed in two steps. In the first step, we bound the Stokes projection at the partition points 
		of the time-intervals. In the second step, we use the exponential interpolant,
		combined with the results of the first step, to obtain a uniform bound on the Stokes projection over $(0,T)$.
		\\
		
		\noindent \textbf{\emph{Step one:}}
		Integrating by parts in time
		in the term containing the temporal derivative in \cref{eq:stokes_proj},
		testing with $\boldsymbol{v}_h = A_h \boldsymbol{\Pi}_h u$ and using \cref{eq:A_h_property1} and \cref{eq:A_h_property2} 
		in~\Cref{lem:disc_lap_properties}, we have
		\begin{equation*}
			\begin{split}
				\frac{1}{2}\int_{I_n} &\od{}{t} a_h(\boldsymbol{\Pi}_h 
				u,\boldsymbol{\Pi}_h u) 
				\dif 
				t + 
				a_h((\boldsymbol{\Pi}_h u)_n^+,(\boldsymbol{\Pi}_h u)_n^+)  +
				\nu \int_{I_n} \norm{A_h \Pi_h u}_{L^2(\Omega)}^2 \\& = \int_{I_n} ( 
				\partial_t u 
				- \nu \Delta  
				u , A \Pi_h u)_{\mathcal{T}_h} \dif t + a_h((\boldsymbol{\Pi}_h u)_n^-,(\boldsymbol{\Pi}_h u)_n^+).
			\end{split}
		\end{equation*}
		Using the coercivity of $a_h(\cdot,\cdot)$, the Cauchy--Schwarz inequality, Young's inequality,
		and summing over all space-time slabs, we find
		\begin{equation*}
			\begin{split}
				\tnorm{(\boldsymbol{\Pi}_h 
					u)_{N}^-}_v^2 &+ 
				\sum_{n=0}^{N-1}\tnorm{[\boldsymbol{\Pi}_h 
					u]_n}_v^2	+ \nu \int_{0}^T \norm{A_h 
					\Pi_h u}_{L^2(\Omega)}^2 \dif t   \\&\le 
				C\del[3]{\frac{1}{\nu} \int_{0}^T \norm{\partial_t u 
						- \nu \Delta  
						u}_{L^2(\Omega)}^2 \dif t + \tnorm{(\boldsymbol{\Pi}_h 
						u)_{0}^-}_v^2}.
			\end{split}
		\end{equation*}
		As $(\Pi_h u_0)^- = P_h u_0$ and
		$(\bar{\Pi}_h u)_0^- = \bar{P}_h u_0$, we have from
		\Cref{lem:approx_prop_spat_proj} and the approximation
		properties of $\bar{P}_h$ (see e.g. \cite{Rhebergen:2018c})
		that
		\begin{equation}\label{eq:low_order_bnd}
			\begin{split}
				\tnorm{(\boldsymbol{\Pi}_h 
					u)_{N}^-}_v^2 + 
				\sum_{n=0}^{N-1}\tnorm{[\boldsymbol{\Pi}_h 
					u]_n}_v^2	+ \nu \int_{0}^T \norm{A 
					\Pi_h u}_{L^2(\Omega)}^2 \dif t     
				\le 
				C\del{\frac{1}{\nu} \int_{0}^T \norm{\partial_t u 
						- \nu \Delta  
						u}_{L^2(\Omega)}^2 \dif t + \norm{u_0}_{H^1(\Omega)}^2}.
			\end{split}
		\end{equation}
		Note
		that for the lowest order scheme $(k=1)$, we can already 
		infer the result.
		\\
		
		\noindent \textbf{\emph{Step two:}}
		It remains to
		obtain a bound for higher order polynomials in time. 
		For this, we use the exponential interpolant of the pair $A_h \boldsymbol{\Pi}_h u = (A_h \Pi_h u, A_h \bar{\Pi}_h u)$, which we denote by
		$\tilde{A}_h \boldsymbol{\Pi}_h u= (\tilde{A}_h \Pi_h u, \tilde{A}_h \bar{\Pi}_h u)$.
		Integrating the first term on the left hand side of \cref{eq:stokes_proj} by parts in time, choosing $\boldsymbol{v}_h = \tilde{A}_h \boldsymbol{\Pi}_h u$, and using \cref{eq:exp_interp_property}, \cref{eq:A_h_property2}, and that $\tilde{A}_h 
		\Pi_h u\in \mathcal{V}_h^{\text{div}}$, we have
		\begin{equation*} \label{eq:stokes_pro_bnd_inter}
			\begin{split}
				\frac{1}{2}\int_{I_n} e^{-\lambda(t-t_n)} \od{}{t} 
				a_h(\boldsymbol{\Pi}_h u,&\boldsymbol{\Pi}_h u) \dif t + 
				((\Pi_h u)_{n}^{+}, (A_h \Pi_h u)_{n}^+)_{\mathcal{T}_h} +
				\nu\int_{I_n} a_h(\boldsymbol{\Pi}_hu , \tilde{A}_h 
				\boldsymbol{\Pi}_h u) 
				\dif t \\  
				&= 
				\int_{I_n}( \partial_t u 
				- \nu \Delta  
				u , \tilde{A}_h \Pi_h u)_{\mathcal{T}_h} \dif t + ((\Pi_h u)_{n}^{-}, 
				(\tilde A_h \Pi_h u)_{n}^+)_{\mathcal{T}_h}.
			\end{split}
		\end{equation*}
		Proceeding in an identical fashion as in the proof of \Cref{lem:linfty_bnd},
		and using \cref{eq:lp_bnd_exp_interp}, we obtain
		\begin{equation}\label{eq:stokes_proj_bnd_inter2}
			\begin{split}
				\frac{e^{-1}C}{2} & 
				\norm{\boldsymbol{\Pi}_h 
					u}_{L^{\infty}(I_n;\boldsymbol{\mathcal{V}}_h)}^2 
				+
				\frac{e^{-1}C}{2} \tnorm{(\boldsymbol{\Pi}_h u )_{n+1}^-}_v^2  \\  
				&\le  C\del[3]{
					\frac{1}{\nu} \int_{I_n} \norm{\partial_t u 
						- \nu \Delta  
						u}_{L^2(\Omega)}^2 \dif t + 
					\nu \int_{I_n} \norm{A \Pi_h
						u}_{L^2(\Omega)}^2 
					\dif t + 
					\tnorm{(\boldsymbol{\Pi}_h u )_{n}^-}_v^2}.
			\end{split}
		\end{equation}
		Bounding the last two terms on the right hand side of \cref{eq:stokes_proj_bnd_inter2}
		using~\cref{eq:low_order_bnd} and omitting
		the second (positive) term on the left hand side, we see that there exists 
		a constant $C>0$ such that
		\begin{equation*}
			\norm{\boldsymbol{\Pi}_h 
				u}_{L^{\infty}(I_n;\boldsymbol{\mathcal{V}}_h)}^2  
			\le C \del{ 
				\frac{1}{\nu} \int_{0}^T \norm{\partial_t u 
					- \nu \Delta  
					u}_{L^2(\Omega)}^2 \dif t  + 
				\norm{u_0}_{H^1(\Omega)}^2}.
		\end{equation*}
		This bound holds uniformly for every space-time slab, so
		the result follows.
	\end{proof}

	\subsubsection{Approximation properties of the parabolic Stokes projection}
	
	\begin{lem} \label{lem:disc_err}
		Let $u$ be the strong velocity solution to the Navier--Stokes system
		\cref{eq:ns_equations} guaranteed by \Cref{thm:strong_solution}, let $(\Pi_h u , \bar{\Pi}_h u) \in \mathcal{V}_h^{\text{div}} \times \bar{\mathcal{V}}_h$ be the velocity pair of
		the Stokes projection \cref{eq:stokes_proj} for $n=0,\dots,N-1$, and let 
		$\mathcal{P}_h$ and 
		$\bar{\mathcal{P}}_h$ 
		denote
		the projections introduced in \Cref{def:space_time_kernel_proj}.
		Let $\zeta_h = \mathcal{P}_h u - \Pi_h u$, $\xi_h = u - \mathcal{P}_h u$,
		$\bar{\zeta}_h = \bar{\mathcal{P}}_h u - \bar{\Pi}_h u$ and $\bar{\xi}_h = 
		u - \bar{\mathcal{P}}_h u$.
		There is a constant $C>0$ such that
		\begin{equation*}
			\norm[1]{ \zeta_h(t_{n+1}^-)}_{L^2(\Omega)}^2 + 
			\sum_{n=0}^{N-1} \norm[1]{ \sbr{\zeta_h}_i}_{L^2(\Omega)}^2 + 
			\nu \int_{0}^{T} \tnorm{\boldsymbol{\zeta}_h}_v^2 \dif t
			\le C \nu 
			\int_0^{T} \tnorm{\boldsymbol{\xi}_h}_{v'}^2  \dif t.
		\end{equation*}
	\end{lem}	
	\begin{proof}
		Our starting point will be the definition of the parabolic Stokes 
		projection~\cref{eq:stokes_proj}. We will introduce the splitting  
		$\boldsymbol{u} - \boldsymbol{\Pi}_h 
		u = \boldsymbol{\xi}_h + \boldsymbol{\zeta}_h$, where 
		$\boldsymbol{\xi}_h = (\xi_h,\bar{\xi}_h) $ and $\boldsymbol{\zeta}_h = 
		(\zeta_h,\bar{\zeta}_h)$. Testing \cref{eq:stokes_proj} with 
		$\boldsymbol{v}_h = \boldsymbol{\zeta}_h \in 
		\mathcal{V}_h^{\text{div}} 
		\times \bar{\mathcal{V}}_h$, integrating by parts in time, 
		using the defining properties of the projection $P_h$ 
		\Cref{def:space_time_kernel_proj}, the coercivity and boundedness of $a_h(\cdot,\cdot)$ \cref{eq:ah_coer_bnd}, the Cauchy--Schwarz inequality
		and Young's inequality with some sufficiently small $\epsilon > 0$, we have
		\begin{equation*}
			\begin{split}
				& 
				\norm[1]{\zeta_h(t_{n+1}^-)}_{L^2(\Omega)}^2  + \norm[1]{ 
					\sbr{\zeta_h}_n}_{L^2(\Omega)}^2 - 
				\norm[1]{\zeta_h(t_{n}^-)}_{L^2(\Omega)}^2
				+ C \nu \int_{I_n} 
				\tnorm{\boldsymbol{\zeta}_h}_v^2 
				\dif t  \le  
				C \nu \int_{I_n}   
				\tnorm{\boldsymbol{\xi}_h}_{v'}^2  \dif t.
			\end{split}
		\end{equation*}
		We conclude by summing over all space-time slabs and noting that $\zeta_h(t_{0}^-) = 0.$
	\end{proof}
	
	\subsection{Error analysis for the velocity}
	
	\subsubsection{The error equation}
	
	We introduce the notation $\boldsymbol{e}_h = (e_h,\bar{e}_h) = (u-u_h,\gamma(u) 
	- \bar{u}_h)$. From \Cref{lem:consistency}, we have the following 
	Galerkin 
	orthogonality
	result:
	\begin{equation} \label{eq:err_eq_inter}
		\begin{split}
			&-\int_{I_n} (e_h, \partial_t v_h)_{\mathcal{T}_h} \dif t + 
			(e_{n+1}^-,v_{n+1}^-)_{\mathcal{T}_h} +
			\nu \int_{I_n} a_h(\boldsymbol{e}_h, \boldsymbol{v}_h) \dif t + 
			\int_{I_n} 
			b_h(\boldsymbol{p} - \boldsymbol{p}_h, v_h)  \dif t \\ &+
			\int_{I_n} \del{ o_h(u; \boldsymbol{u}, \boldsymbol{v}_h) - o_h(u_h; 
				\boldsymbol{u}_h, \boldsymbol{v}_h)} \dif t - 
			(e_{n}^-,v_{n}^+)_{\mathcal{T}_h} 
			= 0, \quad \forall \boldsymbol{v}_h \in \boldsymbol{\mathcal{V}}_h.
		\end{split}
	\end{equation}
	Introducing the splitting $\boldsymbol{e}_h = \del{\boldsymbol{u} - 
		\boldsymbol{\Pi}_h u
	} +  \del{\boldsymbol{\Pi}_h u - \boldsymbol{u}_h} = 
	\boldsymbol{\eta}_h + 
	\boldsymbol{\theta}_h$, integrating by parts in the first term on the left hand 
	side, using the definition of the parabolic Stokes 
	projection~\cref{eq:stokes_proj}, and
	choosing $\boldsymbol{v}_h = \boldsymbol{\theta}_h \in 
	\mathcal{V}_h^{\text{div}} 
	\times 
	\bar{\mathcal{V}}_h$,
	\cref{eq:err_eq_inter} reduces to 
	\begin{equation}\label{eq:error_eq}
		\int_{I_n} (\partial_t \theta_h, \theta_h)_{\mathcal{T}_h} \dif t +
		\nu \int_{I_n} a_h(\boldsymbol{\theta}_h, \boldsymbol{\theta}_h) \dif t + 
		\del{\theta_{n}^+ - \theta_{n}^-,\theta_{n}^+}_{\mathcal{T}_h} =
		-\int_{I_n} \del{ o_h(u; \boldsymbol{u}, \boldsymbol{\theta}_h) - o_h(u_h; 
			\boldsymbol{u}_h, \boldsymbol{\theta}_h)} \dif t,
	\end{equation}
	where we have used that $u_h, \Pi_h u \in P_k(I_n, H)$.
	
	\begin{lem} \label{eq:L2L2_bnd}
		Let $(\Pi_h u , \bar{\Pi}_h u) \in \mathcal{V}_h^{\text{div}} \times \bar{\mathcal{V}}_h$ be the velocity pair of
		the Stokes projection \cref{eq:stokes_proj} and let $(u_h,\bar u_h) \in 
		\boldsymbol{\mathcal{V}}_h$ be an
		approximate velocity solution 
		to the Navier--Stokes system computed using the space-time HDG scheme
		\cref{eq:discrete_problem} for $n= 0, \dots, N-1$.
		Let $\theta_h =  \Pi_h u - u_h $, $\eta_h = u - 
		\Pi_h u$,
		$\bar{\theta}_h = \bar{\Pi}_h u - u_h$ and 
		$\bar{\eta}_h = u - \bar{\Pi}_h u$.
		There exists a constant $C>0$ such that
		
		\begin{equation*} 
			\begin{split}
				&\int_{I_n} \norm{\theta_h}_{L^2(\Omega)}^2 \dif t  \le  
				C \del[2]{ \nu^{1/2} \Delta t^{1/2} \int_{I_n} 
					\tnorm{\boldsymbol{\theta}_h}_v^2 \dif t   +
					\nu \Delta t	 \int_{I_n} 
					\tnorm{\boldsymbol{\eta}_h}_v^2  \dif t}  +
				\Delta t  \norm[1]{\theta_n^-}_{L^2(\Omega)}^2.
			\end{split}
		\end{equation*}
	\end{lem}
	
	\begin{proof}
		We will proceed as in the proof of \cite[Theorem 5.2]{Chrysafinos:2010}.
		Choose $\boldsymbol{z}_h \in V_h^{\text{div}} \times \bar{V}_h$ independent of time.
		We test \cref{eq:err_eq_inter} with the discrete characteristic function $\boldsymbol{z}_{\chi} \in V_h^{\text{div}} \times \bar{V}_h$
		of $\boldsymbol{z}_h$. Recall from \Cref{eq:bnd_char} that we can write $\boldsymbol{z}_{\chi} = \varphi(t) \boldsymbol{z}_h$,
		with $\varphi(t)$ satisfying $\varphi(t_{n}^+) = 1$ as well as \cref{eq:const_int_char} and \cref{eq:const_linf_bnd_char}. Then,
		we have
		\begin{equation} \label{eq:l2_bnd_theta_inter1}
			\begin{split}
				(\theta_h(s), z_h)_{\mathcal{T}_h}  &=
				-\int_{I_n} \del{ o_h(u; \boldsymbol{u}, \boldsymbol{z}_{\chi}) - o_h(u_h; 
					\boldsymbol{u}_h,  \boldsymbol{z}_{\chi}) + \nu 
					a_h(\boldsymbol{\theta}_h, 
					\boldsymbol{z}_{\chi})} \dif t  + \del[1]{ 
					\theta_{n}^-,z_h}_{\mathcal{T}_h}.
			\end{split}
		\end{equation}
		By the boundedness of $a_h(\cdot,\cdot)$ \cref{eq:ah_coer_bnd}, the bound 
		on $\varphi$ 
		\cref{eq:const_linf_bnd_char}, and the Cauchy--Schwarz inequality, 
		\begin{equation} \label{eq:a_const_char_bnd}
			\begin{split}
				\int_{I_n} |a_h(\boldsymbol{\theta}_h, 
				\boldsymbol{z}_{\chi}) |\dif t  \le C\Delta t^{1/2} 
				\tnorm{\boldsymbol{z}_h}_v \del[2]{ \int_{I_n} 
					\tnorm{\boldsymbol{\theta}_h}_v^2 \dif t}^{1/2}.
			\end{split}
		\end{equation}
		After a few algebraic manipulations, we apply \cref{eq:ohLip}, followed by \cref{eq:const_linf_bnd_char}, the energy 
		estimate \Cref{lem:energy_stability}, the assumption 
		\cref{eq:small_data_2} on the problem data, and the Cauchy--Schwarz 
		inequality, to find
		\begin{equation} \label{eq:o_const_char_bnd_1}
			\begin{split}
				&\int_{I_n} | o_h(u; \boldsymbol{u}, \boldsymbol{z}_{\chi}) - o_h(u_h; 
				\boldsymbol{u}_h, \boldsymbol{z}_{\chi})| \dif t \\ 
				& = \int_{I_n} |o_h(u; \boldsymbol{\eta}_h, 
				\boldsymbol{z}_{\chi}) + 
				o_h(\eta_h; \boldsymbol{\Pi}_h u,  \boldsymbol{z}_{\chi}) + o_h(u_h; 
				\boldsymbol{\theta}_h, 
				\boldsymbol{z}_{\chi}) - o_h(\theta_h; \boldsymbol{\Pi}_h u, 
				\boldsymbol{z}_{\chi})| \dif t
				\\&\le
				C \tnorm{\boldsymbol{z}_h}_v \int_{I_n} \del{
					\norm{u}_{H^1(\Omega)}\tnorm{\boldsymbol{\eta}_h}_v 
					+ \tnorm{\boldsymbol{\eta}_h}_v \tnorm{\boldsymbol{\Pi}_h u}_v 
					+ \tnorm{\boldsymbol{u}_h}_v
					\tnorm{\boldsymbol{\theta}_h}_v + 
					\tnorm{\boldsymbol{\Pi}_h u}_v \tnorm{\boldsymbol{\theta}_h}_v 
				} 
				\dif t \\
				& \le C  \nu^{1/2} \tnorm{\boldsymbol{z}_h}_v \del[3]{ \nu^{1/2}\Delta t^{1/2} \del[2]{ 
						\int_{I_n} 
						\tnorm{\boldsymbol{\eta}_h}_v^2  \dif t }^{1/2} +  (\nu^{1/2}\Delta t^{1/2}+1)
					\del[2]{ \int_{I_n} 
						\tnorm{\boldsymbol{\theta}_h}_v^2  \dif t }^{1/2} }.
			\end{split}
		\end{equation}
		Combining \cref{eq:l2_bnd_theta_inter1}, \cref{eq:a_const_char_bnd}, and 
		\cref{eq:o_const_char_bnd_1},
		\begin{equation} \label{eq:l2_bnd_theta_inter2}
			\begin{split}
				(\theta_h(s), z_h)_{\mathcal{T}_h} & \le  C 
				\tnorm{\boldsymbol{z}_h}_v \del[2]{ \nu \Delta t^{1/2} + 
					\nu^{1/2}
				}\del[2]{\int_{I_n} 
					\tnorm{\boldsymbol{\theta}_h}_v^2 \dif t}^{1/2}  \\& \; +
				C
				\tnorm{\boldsymbol{z}_h}_v \nu \Delta t^{1/2}	\del[2]{ \int_{I_n} 
					\tnorm{\boldsymbol{\eta}_h}_v^2  \dif t }^{1/2} +
				\del[1]{
					\theta_{n}^-,z_h}_{\mathcal{T}_h}.
			\end{split}
		\end{equation}
		This holds for any $\boldsymbol{z}_h \in V_h^{\text{div}} \times \bar{V}_h$, so fix 
		$s \in I_n$ and select $\boldsymbol{z}_h =  (\theta_h(s), 
		\bar{\theta}_h(s)) 
		\in \boldsymbol{V}_h$ to find 
		\begin{equation} \label{eq:l2_bnd_theta_inter4}
			\begin{split}
				\norm{\theta_h(s)}_{L^2(\Omega)}^2  \le C \del[2]{ \nu \Delta 
					t^{1/2} + 
					\nu^{1/2}
				}
				\tnorm{\boldsymbol{\theta}_h(s)}_v \del[2]{\int_{I_n} 
					\tnorm{\boldsymbol{\theta}_h}_v^2 \dif t}^{1/2} +  \\
				C \nu \Delta t^{1/2} 
				\tnorm{\boldsymbol{\theta}_h(s)}_v	\del[2]{ \int_{I_n} 
					\tnorm{\boldsymbol{\eta}_h}_v^2  \dif t }^{1/2} +
				\del[1]{\theta_{n}^-,\theta_h(s)}_{\mathcal{T}_h}.
			\end{split}
		\end{equation}
		This holds for all $s \in I_n$, so the result follows after integrating both sides over $I_n$ 
		and applying the
		Cauchy--Schwarz inequality and Young's inequality. 
	\end{proof}

	\begin{lem} \label{lem:err_eq_bnd}
		Let $u \in L^{\infty}(0,T;V)\cap L^2(0,T;V\cap 
		H^2(\Omega)^d) \cap H^1(0,T;H)$ be the strong solution to the continuous 
		Navier--Stokes problem, let $(\Pi_h u , \bar{\Pi}_h u) \in \mathcal{V}_h^{\text{div}} \times \bar{\mathcal{V}}_h$ be the velocity pair of
		the Stokes projection \cref{eq:stokes_proj}, and let $(u_h,\bar u_h) \in 
		\boldsymbol{\mathcal{V}}_h$ be an
		approximate velocity solution 
		to the Navier--Stokes system computed using the space-time HDG scheme
		\cref{eq:discrete_problem} for $n= 0, \dots, N-1$.
		Let $\theta_h = \Pi_h u - u_h$, $\eta_h = u - 
		\Pi_h u$,
		$\bar{\theta}_h = \bar{\Pi}_h u - \bar{u}_h$ and 
		$\bar{\eta}_h = u - \bar{\Pi}_h u$.
		There exists a constant $C>0$ such that
		\begin{equation*}
			\begin{split}
				\norm[1]{\theta_{N}^-}_{L^2(\Omega)}^2 + \sum_{n=0}^{N-1}
				\norm[1]{ 
					\sbr{\theta_h}_n}_{L^2(\Omega)}^2   &+
				\nu \int_{0}^{T} \tnorm{ 
					\boldsymbol{\theta}_h}_v^2 \dif t 
				\le C\exp \del[2]{CT}
				\nu \Delta t \int_{0}^{T} 
				\tnorm{\boldsymbol{\eta}_h}_v^2 \dif t,
			\end{split} 
		\end{equation*}
		provided the time step satisfies $\Delta t \le C\nu$.
	\end{lem}
	
	\begin{proof}
		Our starting point for deriving an error estimate for the velocity will be 
		the 
		error equation~\cref{eq:error_eq}.
		We begin by bounding the nonlinear convection terms. A few algebraic 
		manipulations yield
		\begin{equation*}
			\begin{split}
				&-\int_{I_n}  ( o_h(u; \boldsymbol{u}, 
				\boldsymbol{\theta}_h) - 
				o_h(u_h; 
				\boldsymbol{u}_h, \boldsymbol{\theta}_h)) \dif t \\
				& \le \int_{I_n}|o_h(u; \boldsymbol{\eta}_h, 
				\boldsymbol{\theta}_h)|\dif t+ 
				\int_{I_n}|o_h(\eta_h; \boldsymbol{\Pi}_h u, 
				\boldsymbol{\theta}_h)| \dif t+ 
				\int_{I_n}|o_h(\theta_h; 
				\boldsymbol{\Pi}_h u, 
				\boldsymbol{\theta}_h)|\dif t = T_1 + T_2 + T_3,
			\end{split}
		\end{equation*}
		where we have used that $o_h(u_h; \boldsymbol{\theta}_h, 
		\boldsymbol{\theta}_h) \ge 0$.
		We now bound $T_1$ and $T_2$. By \cref{eq:ohLip},
		the assumption 
		\cref{eq:small_data_2} on the problem data, and Young's inequality with some $\epsilon_1 > 0$,
		we find
		\begin{equation} \label{eq:theta_nonlinear_1}
			\begin{split}
				\int_{I_n} |o_h(u; \boldsymbol{\eta}_h, \boldsymbol{\theta}_h)| 
				\dif t
				&\le \frac{C\nu}{2\epsilon_1} 
				\int_{I_n} \tnorm{\boldsymbol{\eta}_h}_v^2 \dif t + 
				\frac{\nu\epsilon_1}{2} 
				\int_{I_n} \tnorm{\boldsymbol{\theta}_h}_v^2 \dif t,
			\end{split}
		\end{equation}
		and similarly,
		\begin{equation} \label{eq:theta_nonlinear_2}
			\begin{split}
				\int_{I_n}| o_h(\eta_h; \boldsymbol{\Pi}_h \boldsymbol{u}, 
				\boldsymbol{\theta}_h) |\dif t
				&\le \frac{C\nu}{2\epsilon_1}  \int_{I_n} 
				\tnorm{\boldsymbol{\eta}_h}_v^2 \dif t + \frac{\nu\epsilon_1}{2}
				\int_{I_n} 
				\tnorm{\boldsymbol{\theta}_h}_v^2 \dif t.
			\end{split}
		\end{equation}
		The bound on $T_3$ is more complicated. To begin, we use
		\Cref{lem:o_space_bnd} and H\"{o}lder's inequality with $p=4$ and $q= 4/3$
		to find
		\begin{equation*}
			\begin{split}
				\int_{I_n}|o_h(\theta_h; 
				\boldsymbol{\Pi}_h u, 
				\boldsymbol{\theta}_h)|\dif t
				& \le C \del[2]{\int_{I_n} 
					\norm{\theta_h}_{L^2(\Omega)}^{2} \tnorm{\boldsymbol{\Pi}_h u}_v^4 
					\dif t 
				}^{1/4} \del[2]{ 
					\int_{I_n} \tnorm{\boldsymbol{\theta}_h}_v^{2} \dif t}^{3/4}.
			\end{split}
		\end{equation*}
		Recall Young's inequality in the form
		$ab \le \epsilon_2^{p/q}a^p/p + b^q/(q \epsilon_2)$ where
		$1/p+1/q=1$, $1 < q,p < \infty$, $a, b > 0$, and
		$\epsilon_2 >0$ (see e.g. \cite[Appendix
		A]{John:book}). Choosing $p=4$ and $q = 4/3$ we find
		\begin{equation} \label{eq:T_3_bnd}
			\begin{split}
				\int_{I_n}|o_h(\theta_h; 
				\boldsymbol{\Pi}_h u, 
				\boldsymbol{\theta}_h)|\dif t
				& \le C \del[3]{\frac{\epsilon_2^{3}}{4}  
					\nu^4 \int_{I_n} 
					\norm{\theta_h}_{L^2(\Omega)}^{2} \dif t  
					+ \frac{3}{4 \epsilon_2}
					\int_{I_n} \tnorm{\boldsymbol{\theta}_h}_v^{2} \dif t}.
			\end{split}
		\end{equation}
		Here, we have used
		the uniform bound on the Stokes projection in \Cref{lem:stokes_proj_bnd} and the 
		assumption \cref{eq:small_data_1} on the problem data.
		Next, we consider the error 
		equation \cref{eq:error_eq}.
		Integrating by parts in time on the left hand side of \cref{eq:error_eq},
		combining the result with~\cref{eq:theta_nonlinear_1},~\cref{eq:theta_nonlinear_2},
		\cref{eq:T_3_bnd}, and using the coercivity of $a_h(\cdot,\cdot)$ 
		\cref{eq:ah_coer_bnd}, we have for some constants $C_1,C_2 > 0$:
		\begin{equation} \label{eq:pre_L2_bound_inter}
			\begin{split}
				&	\norm[1]{\theta_{n+1}^-}_{L^2(\Omega)}^2 + \norm[1]{ 
					\sbr{\theta_h}_n}_{L^2(\Omega)}^2 - 
				\norm[1]{\theta_{n}^-}_{L^2(\Omega)}^2 +
				C_1\nu \int_{I_n} \tnorm{ \boldsymbol{\theta}_h}_v^2 \dif t \\ 
				&\le C_2\del[3]{
					\nu 
					\epsilon_1^{-1}\int_{I_n} 
					\tnorm{\boldsymbol{\eta}_h}_v^2 \dif t + 	\epsilon_1\nu
					\int_{I_n} 
					\tnorm{\boldsymbol{\theta}_h}_v^2 \dif t +  
					\epsilon_2^{3} \nu^{4} \int_{I_n} 
					\norm{\theta_h}_{L^2(\Omega)}^{2} \dif t  
					+ \epsilon_2^{-1}
					\int_{I_n} \tnorm{\boldsymbol{\theta}_h}_v^{2} \dif t }.
			\end{split}
		\end{equation}
		Choosing $\epsilon_1 = C_1/(2C_2)$ and $\epsilon_2 = C_3\nu^{-1} $ 
		where $C_3 > 2C_2/C_1$ in 
		\cref{eq:pre_L2_bound_inter}, letting $C_4 = C_1/2 - C_2/C_3 > 0$, and using \Cref{eq:L2L2_bnd}, we have upon 
		rearranging that
		\begin{equation*} 
			\begin{split}
				\norm[1]{\theta_{n+1}^-}_{L^2(\Omega)}^2 &+ \norm[1]{ 
					\sbr{\theta_h}_n}_{L^2(\Omega)}^2 +
				\del{C_4\nu - C_5\nu^{1/2}\Delta t^{1/2}} \int_{I_n} \tnorm{ 
					\boldsymbol{\theta}_h}_v^2 \dif t \\ &\le 
				\del{ 1 + C_5 \Delta t} 
				\norm[1]{\theta_{n}^-}_{L^2(\Omega)}^2  +
				C_5\nu\Delta t\int_{I_n} 
				\tnorm{\boldsymbol{\eta}_h}_v^2 \dif t. 
			\end{split}
		\end{equation*}
		Summing over all space-time slabs and noting that $\theta_{0}^- = 0$, we have
		\begin{equation*} 
			\begin{split}
				\norm[1]{\theta_{N}^-}_{L^2(\Omega)}^2 &+ \sum_{n=0}^{N-1} 
				\norm[1]{ 
					\sbr{\theta_h}_n}_{L^2(\Omega)}^2  +
				\del{C_4\nu - C_5\nu^{1/2}\Delta t^{1/2}} \int_{0}^{T} \tnorm{ 
					\boldsymbol{\theta}_h}_v^2 \dif t
				\\ &\le 
				C_5  \Delta t \del{\sum_{n=0}^{N-1}
					\norm[1]{\theta_{i}^-}_{L^2(\Omega)}^2  +
					\nu \int_{0}^{T} 
					\tnorm{\boldsymbol{\eta}_h}_v^2 \dif t}.
			\end{split}
		\end{equation*}
		The result follows by a discrete Gr\"onwall inequality
		\cite[Lemma 1.11]{Dolejsi:book} for $\Delta t < C_4\nu/(2C_5)$
		and using that
		$ \prod_{j=0}^{N-1} \del[0]{1 + C \Delta t} \le \exp
		\del[0]{C\sum_{j=0}^{N-1}\Delta t} \le \exp \del[0]{CT}$.
	\end{proof}
	\subsection{Proof of \Cref{thm:vel_error_estimates}}
	\begin{proof}
		Let $\boldsymbol{e}_h = \boldsymbol{u} - \boldsymbol{u}_h$. We introduce the splitting $\boldsymbol{e}_h =\boldsymbol{\xi}_h + \boldsymbol{\zeta}_h + \boldsymbol{\theta}_h$,
		where $\boldsymbol{\theta}_h = \boldsymbol{\Pi}_h u - \boldsymbol{u}_h$, 
		$\boldsymbol{\zeta}_h = \boldsymbol{\mathcal{P}}_h u - \boldsymbol{\Pi}_h u$, and $\boldsymbol{\xi}_h  = \boldsymbol{u} - \boldsymbol{\mathcal{P}}_h u$.
		Using the triangle inequality, \Cref{lem:err_eq_bnd}, \Cref{lem:disc_err}, and noting that
		$\sbr{\xi_h}_n = 0$ for $n=0,\dots,N-1$, we find
		there exists a constant $C>0$ such that
		\begin{multline}
			\label{eq:err_bnd_inter2}
			\norm[1]{e_{N}^-}_{L^2(\Omega)}^2  + \sum_{n=0}^{N-1} 
			\norm{\sbr{e_h}_n}_{L^2(\Omega)}^2 + \sum_{n=0}^{N-1}
			\nu \int_{0}^{T} \tnorm{ \boldsymbol{e}_h}_{v'}^2 \dif t 
			\\
			\le \exp 
			\del[2]{CT}\del{ 
				\norm[1]{\xi_{N}^-}_{L^2(\Omega)}^2  +	
				\nu
				\int_0^{T} \tnorm{\boldsymbol{\xi}_h}_{v'}^2  \dif t}.
		\end{multline}
		To bound the last term on the right hand side of 
		\cref{eq:err_bnd_inter2}, we employ \Cref{thm:projection_estimates} to
		find 
		\begin{equation}
			\int_0^{T} \tnorm{\boldsymbol{\xi}_h}_{v'}^2  \dif t \lesssim h^{2k} 
			\norm{u}_{L^2(0,T,H^{k+1}(\Omega))}^2 + \Delta t^{2k+2} 
			\norm{u}_{H^{k+1}(0,T,H^2(\Omega))}^2.
		\end{equation}
		The result will follow after bounding $\norm[0]{\xi_{N}^-}_{L^2(\Omega)}$. By 
		\Cref{lem:approx_prop_spat_proj}, there exists a constant $C>0$ such 
		that
		\begin{equation*}
			\norm[1]{\xi_{N}^-}_{L^2(\Omega)}^2 =	\norm[1]{u(T) - 
				(P_{h} u) 
				(T)}_{L^2(\Omega)}^2 \lesssim h^{2k+2} 
			\norm{u}_{C(0,T;H^{k+1}(\Omega))}^2.
		\end{equation*}
	\end{proof}

	\section{Numerical results}
	\label{sec:numerical_results}
	
	In this section, we consider a simple test case with a manufactured solution
	to verify the theoretical results of the previous sections. We 
	solve the 
	Navier--Stokes equations on the space-time domain 
	$\Omega \times [0,T] = 
	[0,1]^3$. We impose Dirichlet boundary conditions along the boundaries $x = 0$, 
	$x = 1$, $y = 0$, and Neumann boundary conditions along $y = 1$. We choose the 
	problem data such that the exact solution is given by
	\begin{equation*}
		u = 
		\begin{bmatrix}
			2 + \sin(2\pi (x - t)) \sin(2\pi (y - t))\\
			2 + \cos(2\pi (x - t)) \cos(2\pi (y - t))
		\end{bmatrix},
		\quad
		p = \sin(2\pi (x - t)) \cos(2\pi (y - t)).
	\end{equation*}
	This example was implemented using the Modular Finite Element Methods (MFEM) 
	library \cite{mfem,mfem-web} on prismatic space-time meshes. 
	
	We present the velocity and pressure errors, measured in the mesh-dependent 
	$\tnorm{\cdot}_{v'}$-norm and $\norm{\cdot}_{L^2(0,T;L^2(\Omega))}$-norm, 
	respectively, and rates of convergence for different levels of space-time 
	refinement with polynomial degrees $k=2$ and $k=3$ in \Cref{tab:rates}. We 
	observe optimal rates of convergence for the velocity in the 
	$\tnorm{\cdot}_{v'}$ norm as expected from 
	\Cref{thm:vel_error_estimates} as well as optimal rates of convergence for the 
	pressure in the 
	$L^2(0,T;L^2(\Omega))$-norm.
	\begin{center}
		\begin{table}[h]
			\centering
			\begin{tabular}{cc|cc|cc}
				Cells per slab  & Nr. of slabs & 
				$\tnorm{\boldsymbol{u-u}_h}_{v'}$ & Rate 
				& $\|p-p_h\|_{L^2(\Omega \times [0,T])}$ & Rate \\ \hline
				128 &  20 & 8.6e-01 &  -  & 7.9e-03 &  -   \\
				512 &  40 & 2.1e-01 & 2.0 & 2.6e-03 & 1.6  \\
				2048 &  80 & 5.2e-02 & 2.0 & 6.7e-04 & 1.9  \\
				8192 & 160 & 1.3e-02 & 2.0 & 1.7e-04 & 2.0  \\\hline
				128 &  20 & 2.0e-01 &  -  & 6.9e-04 &  -  \\
				512 &  40 & 2.7e-02 & 2.9 & 5.2e-05 & 3.7 \\
				2048 &  80 & 3.5e-03 & 3.0 & 4.7e-06 & 3.5 \\
				8192 & 160 & 4.3e-04 & 3.0 & 5.1e-07 & 3.2 \\
			\end{tabular}
			\caption{Rates of convergence when solving \cref{eq:ns_equations} with 
				$\nu = 10^{-4}$. Note that $\Delta t = 1/(\text{Nr. of slabs})$. Top: 
				using polynomials of degree $k = 2$, bottom: using 
				polynomials of degree $k = 3$.}\label{tab:rates}
		\end{table}
	\end{center}
	\section{Conclusion}
	\label{sec:conclusion}
	In this paper we have analyzed a space-time hybridized discontinuous
	Galerkin method for the evolutionary Navier--Stokes equations that is
	pointwise mass conserving, energy-stable, and pressure-robust. We have
	proven that there exists a solution to the resulting nonlinear
	algebraic system of equations in both two and three spatial
	dimensions, and under a small data hypothesis the solution is unique
	in two spatial dimensions. A priori error estimates in a
	mesh-dependent energy norm for the velocity which are anisotropic in
	space and time and independent of the pressure have been derived.
	Finally, we have verified our theoretical error bounds for the
	velocity with a numerical example.
	
	\appendix
	
	\renewcommand{\thesection}{\Alph{section}}
	
	\section{Projection estimates} \label{appendix:projection_estimates}
	Here, we briefly outline the approximation 
	properties of the projections $\mathcal{P}_h$ and $ \bar{\mathcal{P}}_h$ 
	introduced
	in \Cref{def:space_time_kernel_proj}.
	We require some results from \cite{Dolejsi:book} adapted to the 
	present setting:
	\begin{thm}{\cite[Theorem 6.9]{Dolejsi:book}} \label{Thm:interchange_proj}
		The projections defined in \Cref{def:space_time_kernel_proj} exist and are 
		unique. 
		Furthermore, for all $n=0,\dots, N-1$,
		\begin{equation*}
			(\mathcal{P}_h v)|_{I_n} = \mathcal{P}_h(P_h v)|_{I_n} = P_h (\mathcal{P}_h 
			v|_{I_n}), \quad
			(\bar{\mathcal{P}}_h v)|_{I_n} = \bar{\mathcal{P}}_h(\bar{P}_h v)|_{I_n} = 
			\bar{P}_h (\bar{\mathcal{P}}_h 
			v|_{I_n}).
		\end{equation*}
	\end{thm}
	\noindent
	We first introduce a temporal ``DG-projection"
	$P^t:C(I_n) \to P_k(I_n)$ satisfying
	$\int_{I_n} \del{P^t w(t) - w(t)} v \dif t = 0$ for all
	$v \in P_{k-1}(I_n)$ and such that $P^t w(t_{n+1}^-) = w(t_{n+1}^-)$
	(see, e.g., \cite[Chapter 12]{Thomee:book} or \cite[Lemma
	6.11]{Dolejsi:book}). For $u \in H^{k+1}(I_n)$ this projection
	satisfies
	\begin{equation} \label{eq:time_proj_estimates}
		\norm[0]{u - P^t u}_{H^s(I_n)} \lesssim \Delta t^{r-s} 
		|u|_{H^r(I_n)}, \quad 0 \le s \le 1 \le r \le k.
	\end{equation}

	\begin{lem} \label{lem:proj_interchange}
		Let $\varphi \in C(I_n;V_h^{\text{div}})$ and $\bar \psi \in C(I_n;\bar 
		V_h)$. Then
		\begin{subequations}
			\begin{align}
				\label{eq:proj_inter_1}
				\mathcal{P}_h \varphi(x,t) &= P^t \varphi(x,t) && \forall x\in \cbr{K,F},\ \forall \cbr{K \in \mathcal{T}_h, F \in \mathcal{F}_h},
				\\
				\label{eq:proj_inter_2}
				\bar{\mathcal{P}}_h \bar \psi(x,t) &= P^t\bar \psi(x,t) && \forall x \in F,\ \forall F \in \mathcal{F}_h.
			\end{align}
		\end{subequations}
	\end{lem}
	\begin{proof}
		The proofs of \cref{eq:proj_inter_1} and \cref{eq:proj_inter_2} follow the proof of \cite[Lemma 
		6.11]{Dolejsi:book} with minor modifications.
	\end{proof}
	\noindent With \Cref{lem:proj_interchange} in hand,
	we can prove the following results:
	\begin{thm} \label{thm:projection_estimates}
		Let $k \ge 1$, $m \in \cbr{0,1}$, and $u \in H^{k+1}(I_n;H_0^1(\Omega)^d\cap 
		H^2(\Omega)^d) \cap C(I_n;H^{k+1}(\Omega)^d)$. Let $\mathcal{P}_h$
		and $\bar{\mathcal{P}}_h$ be the projections defined in \Cref{def:space_time_kernel_proj}.  
		Let $l = 0$ if $m \le 1$ and $l = m$ if $m=2$. Then, the following estimates 
		hold:
		\begin{subequations}
			\begin{align} 
				\sum_{K \in \mathcal{T}_h} \int_{I_n} & |u - \mathcal{P}_h 
				u|_{H^m(K)}^2 \dif t \label{eq:spatial_der_estimate} \\ 
				&\lesssim h^{2\del{k-m+1}} 
				\norm{u}_{L^2(I_n;H^{k+1}(\Omega))}^2 + h^{-l}\Delta 
				t^{2k+2}\norm{u}_{H^{k+1}(I_n;H^m(\Omega))}^2 ,
				\notag
				\\
				\label{eq:proj_bnd_faces}
				\sum_{K\in \mathcal{T}_h} h_K^{-1} \int_{I_n}& \norm{ 
					\mathcal{P}_h  u - \bar{\mathcal{P}}_hu}_{L^2(\partial K)}^2 \dif 
				t \lesssim
				h^{2k} 
				\norm{u}_{L^2(I_n;H^{k+1}(\Omega))}^2, \\ 	\label{eq:proj_bnd_grad_faces}
				\sum_{K\in \mathcal{T}_h} h_K \int_{I_n} & \norm{ \nabla (u -
					\mathcal{P}_h u)n}_{L^2(\partial K)}^2 \dif 
				t \\ & \lesssim \notag
				h^{2k} 
				\norm{u}_{L^2(I_n;H^{k+1}(\Omega))}^2 + \Delta 
				t^{2k+2}\norm{u}_{H^{k+1}(I_n;H^2(\Omega))}^2.
			\end{align}
		\end{subequations}
	\end{thm}
	\begin{proof}
		First, \cref{eq:spatial_der_estimate} can be shown in a similar fashion to 
		\cite[Lemmas 6.17 and 6.18]{Dolejsi:book} using the approximation properties of 
		the spatial projection $P_h$ given in \Cref{lem:approx_prop_spat_proj} and the 
		approximation 
		properties of $P^t$ given in \cref{eq:time_proj_estimates}. 
		For \cref{eq:proj_bnd_faces} we recall that $(\mathcal{P}_h v)|_{I_n} = 
		\mathcal{P}_h(P_h v)|_{I_n}$ and $(\bar{\mathcal{P}}_h v)|_{I_n} = 
		\bar{\mathcal{P}}_h(\bar{P}_h 
		v)|_{I_n}$
		by \Cref{Thm:interchange_proj}, so by \Cref{lem:proj_interchange}, Fubini's 
		theorem, and the stability of 
		$P^t$ in the $L^2(I_n)$ norm, we have
		\begin{equation*}
			\begin{split}
				\sum_{K\in \mathcal{T}_h} \frac{1}{h_K} \int_{I_n} \norm[0]{ 
					\mathcal{P}_h  u - \bar{\mathcal{P}}_hu}_{L^2(\partial K)}^2 \dif t
				&\le C\sum_{K\in \mathcal{T}_h} \frac{1}{h_K} \int_{I_n} \norm[0]{ 
					P_h  u - \bar P_h u}_{L^2(\partial 
					K)}^2\dif t.
			\end{split}
		\end{equation*}
		We conclude using the triangle inequality and the approximation properties 
		of the spatial projections $P_h$ and $\bar{P}_h$. Finally, \cref{eq:proj_bnd_grad_faces}
		follows from \cref{eq:spatial_der_estimate} after noting that a local
		trace inequality yields
		\begin{equation*}
			h_K \norm{ \nabla (u -
				\mathcal{P}_h u)n}_{L^2(\partial K)}^2 \le |u -
			\mathcal{P}_h u|^2_{H^1(K)} + h_K^2 |u -
			\mathcal{P}_h u|^2_{H^2(K)}.
		\end{equation*}		
	\end{proof}
	
	\bibliographystyle{amsplain.bst}
	\bibliography{references}
\end{document}